\documentclass[a4paper,11pt]{article}%
\usepackage{amsfonts}
\usepackage{amsmath}
\usepackage{amssymb}
\usepackage{graphicx}%
\providecommand{\U}[1]{\protect\rule{.1in}{.1in}}
\newtheorem{theorem}{Theorem}
\numberwithin{theorem}{subsection}

\newtheorem{corollary}[theorem]{Corollary}

\newtheorem{definition}[theorem]{Definition}

\newtheorem{lemma}[theorem]{Lemma}

\newtheorem{proposition}[theorem]{Proposition}
\newtheorem{remark}[theorem]{Remark}

\newenvironment{proof}[1][Proof]{\noindent\textbf{#1.} }{\ \rule{0.5em}{0.5em}}
\begin{document}

\title{Fractional Sobolev inequalities associated with singular problems}
\author{{\small \textbf{Grey Ercole\thanks{Email: grey@mat.ufmg.br (corresponding
author)} \ and Gilberto de Assis Pereira \thanks{Email:
gilbertoapereira@yahoo.com.br}}}\\{\small \textit{Departamento de Matem\'{a}tica - Universidade Federal de Minas
Gerais}}\\{\small \textit{Belo Horizonte, MG, 30.123-970, Brazil.}}}
\maketitle

\begin{abstract}
\noindent In this paper we consider Sobolev inequalities associated with
singular problems for the fractional $p$-Laplacian operator in a bounded
domain of $\mathbb{R}^{N}$, $N\geq2.\bigskip$



\end{abstract}

\section{Introduction}

Let $\Omega$ be a bounded, smooth domain of $\mathbb{R}^{N}$ ($N\geq2$) and,
for $0<s<1<p<\infty,$ let $W_{0}^{s,p}(\Omega)$ denote the fractional Sobolev
space defined as the completion of $C_{c}^{\infty}(\Omega)$ with respect to
the norm%
\begin{equation}
u\mapsto\left(  \left[  u\right]  _{s,p}^{p}+\left\Vert u\right\Vert _{p}%
^{p}\right)  ^{\frac{1}{p}}, \label{norma2}%
\end{equation}
where%
\begin{equation}
\left[  u\right]  _{s,p}:=\left(  \int\int_{\mathbb{R}^{2N}}\frac{\left\vert
u(x)-u(y)\right\vert ^{p}}{\left\vert x-y\right\vert ^{N+sp}}\mathrm{d}%
x\mathrm{d}y\right)  ^{\frac{1}{p}} \label{Gag}%
\end{equation}
is the Gagliardo semi-norm and $\left\Vert \cdot\right\Vert _{r}$ denotes the
standard norm of $L^{r}(\Omega),$ $1\leq r\leq\infty$ (a notation that will be
used in the whole paper).

Thanks to the fractional Poincar\'{e} inequality (see \cite[Lemma 2.4]{BLP}),
\begin{equation}
\left\Vert u\right\Vert _{p}^{p}\leq C_{N,s,p,\Omega}\left[  u\right]
_{s,p}^{p},\quad\forall\,u\in C_{c}^{\infty}(\Omega), \label{poincare}%
\end{equation}
$\left[  \cdot\right]  _{s,p}$ is a norm in $W_{0}^{s,p}(\Omega)$ equivalent
to (\ref{norma2}). Thus,
\[
W_{0}^{s,p}(\Omega)=\left\{  u\in L^{p}(\mathbb{R}^{N}):\left[  u\right]
_{s,p}<\infty\quad\mathrm{and}\quad u=0\quad\mathrm{in}\quad\mathbb{R}%
^{N}\setminus\Omega\right\}
\]
equipped with the norm $\left[  \cdot\right]  _{s,p}$ is a Banach space.
Moreover, $W_{0}^{s,p}(\Omega)$ is uniformly convex and compactly embedded
into $L^{r}(\Omega),$ for all
\[
1\leq r<p_{s}^{\star}:=\left\{
\begin{array}
[c]{lcl}%
\frac{Np}{N-sp} & \mathrm{if} & N>sp\\
\infty & \mathrm{if} & N\leq sp,
\end{array}
\right.
\]
continuously embedded into $L^{p_{s}^{\star}}(\Omega)$ when $N>sp,$ and
compactly embedded into the H\"{o}lder space $C^{s-\frac{N}{p}}(\Omega)$ when
$N<sp$ (see \cite[Lemma 2.9]{BLP}). We refer the reader to \cite{Guide} for a
self-contained exposition on the fractional Sobolev spaces.

In this paper we will consider the Sobolev inequalities associated with the
fractional, singular problem%
\begin{equation}
\left\{
\begin{array}
[c]{lcl}%
(-\Delta_{p})^{s}\,u=\dfrac{\omega}{u^{\alpha}} & \mathrm{in} & \Omega\\
u>0 & \mathrm{in} & \Omega\\
u=0 & \mathrm{on} & \mathbb{R}^{N}\setminus\Omega,
\end{array}
\right.  \label{palfa}%
\end{equation}
where $0<\alpha\leq1,$ $\omega$ is a nonnegative (weight) function in
$L^{r}(\Omega)\setminus\left\{  0\right\}  ,$ for some $r\geq1,$ and
$(-\Delta_{p})^{s}$ denotes the fractional $p$-Laplacian, formally defined by
\[
(-\Delta_{p})^{s}\,u(x)=\lim_{\epsilon\rightarrow0^{+}}\int_{\mathbb{R}%
^{N}\setminus B_{\epsilon}(x)}\frac{|u(x)-u(y)|^{p-2}(u(x)-u(y))}%
{|x-y|^{N+sp}}\mathrm{d}y.
\]

In the case $0<\alpha<1$ the Sobolev inequality associated with (\ref{palfa})
takes the form
\begin{equation}
C\left(
{\displaystyle\int_{\Omega}}
\left\vert v\right\vert ^{1-\alpha}\omega\mathrm{d}x\right)  ^{\frac
{p}{1-\alpha}}\leq\left[  v\right]  _{s,p}^{p},\quad\forall\,v\in W_{0}%
^{s,p}(\Omega). \label{Sobolev1}%
\end{equation}

We will prove that the best (\textrm{i.e.} the larger) constant $C$ in
(\ref{Sobolev1}) is
\[
\lambda_{\alpha}=\left[  u_{\alpha}\right]  _{s,p}^{p(\frac{1-\alpha
-p}{1-\alpha})},
\]
where $u_{\alpha}$ denotes the only weak solution of (\ref{palfa}). We also
will show that
\[
\lambda_{\alpha}\left(
{\displaystyle\int_{\Omega}}
\left\vert v\right\vert ^{1-\alpha}\omega\mathrm{d}x\right)  ^{\frac
{p}{1-\alpha}}=\left[  v\right]  _{s,p}^{p}%
\]
if, and only if, $v$ is a scalar multiple of $u_{\alpha}.$

By means of a limit procedure (when $\alpha\rightarrow1^{-}$) we will deduce
the following Sobolev inequality%
\begin{equation}
C\exp\left(  \tfrac{p}{\left\Vert \omega\right\Vert _{1}}%
{\displaystyle\int_{\Omega}}
(\log\left\vert v\right\vert )\omega\mathrm{d}x\right)  \leq\left[  v\right]
_{s,p}^{p},\quad\forall\,v\in W_{0}^{s,p}(\Omega). \label{Sobolev2}%
\end{equation}
Moreover, we will prove that the best constant $C$ in this inequality is
\begin{equation}
\mu:=\lim_{\alpha\rightarrow1^{-}}\lambda_{a}\left\Vert \omega\right\Vert
_{1}^{\frac{p}{1-\alpha}}, \label{mialfa}%
\end{equation}
provided that it is finite, and that
\[
\mu\exp\left(  \tfrac{p}{\left\Vert \omega\right\Vert _{1}}%
{\displaystyle\int_{\Omega}}
(\log\left\vert v\right\vert )\omega\mathrm{d}x\right)  =\left[  v\right]
_{s,p}^{p}%
\]
if, and only if, $v$ is a scalar multiple of the only weak solution of the
singular problem
\[
\left\{
\begin{array}
[c]{lcl}%
(-\Delta_{p})^{s}\,u=\dfrac{\mu}{\left\Vert \omega\right\Vert _{1}}%
\dfrac{\omega}{u} & \mathrm{in} & \Omega\\
u>0 & \mathrm{in} & \Omega\\
u=0 & \mathrm{on} & \mathbb{R}^{N}\setminus\Omega.
\end{array}
\right.
\]

Our approach here is based on that developed in \cite{EP}, where we have
considered the local, singular equation $-\operatorname{div}\left(  \left\vert
\nabla u\right\vert ^{p-2}\nabla u\right)  =u^{-1}.$ Here, besides the
technical difficulties related to the nonlocal operator, we also have to deal
with a non-constant weight $\omega\in L^{r}(\Omega).$

The literature on singular problems for equations of the form $\mathcal{L}%
u=\omega u^{-\alpha}$ has primarily focused on local operators as the
Laplacian, $\mathcal{L}u=-\operatorname{div}\nabla u$ (see \cite{AM,
Bocorsina, CD, CRT, LS, LM, Stuart}), or the $p$-Laplacian, $\mathcal{L}%
u=-\operatorname{div}\left(  \left\vert \nabla u\right\vert ^{p-2}\nabla
u\right)  ,$ $p>1$ (see \cite{Anello, CST, EP, GST, GST1, Mohammed}).

As regarding to nonlocal (fractional) operators, the literature on singular
problems is quite recent and more restricted to $\mathcal{L}u=(-\Delta
_{p})^{s}\,u$ (see \cite{BBMP, CMS}). Furthermore, according to our knowledge,
Sobolev-type inequalities associated with fractional singular problems have
not been investigated up to now.

In general, the energy functional associated with a singular problem of the
form $\mathcal{L}u=\omega u^{-\alpha}$ is not differentiable. This fact makes
very difficult the direct application of variational methods for proving
existence of solutions for this kind of problem. In order to overcome this
issue (in the cases where $\mathcal{L}$ is a local operator), authors have
employed the sub-super solutions method (see \cite{CD, LM, Mohammed}) or a
method of approximation by nonsigular problems introduced in \cite{Bocorsina}
by Boccardo and Orsina (see \cite{AM, CST}). Recently, in \cite{CMS}, the
latter method was applied to (\ref{palfa}) in order to obtain the existence of
a weak solution, in the case $0<\alpha\leq1,$ and also the existence of a
solution in $W_{loc}^{1,s}(\Omega),$ in the case $\alpha>1.$ We remark that
singular problems for equations of the form $\mathcal{L}u=\omega u^{-\alpha}$
might not have weak solutions (in the standard sense) when $\alpha>1$ and
$\omega$ is a general positive weight (see \cite{LM}). This fact is related to
the singularity of the problem when the support of $\omega$ intercepts the
boundary $\partial\Omega.$ In fact, if $\alpha>1$ and the support of $\omega$
is contained in a proper subdomain of $\Omega,$ the singular
problem\ (\ref{palfa}) has a unique weak solution (see Remark \ref{alfa>1}).

In order to make this paper self-contained we will present, in Section
\ref{Sec2}, results of existence, uniqueness and boundedness (in $L^{\infty}$)
for the singular problem (\ref{palfa}). The existence will be proved by
applying the approximation method by Boccardo and Orsina, which consists in
finding a solution as the limit of the sequence $\left\{  u_{n}\right\}
_{n\in\mathbb{N}}\subset W_{0}^{s,p}(\Omega)$ satisfying
\[
\left\{
\begin{array}
[c]{lcl}%
(-\Delta_{p})^{s}\,u_{n}=\dfrac{\omega_{n}}{(u_{n}+\frac{1}{n})^{\alpha}} &
\mathrm{in} & \Omega\\
u_{n}>0 & \mathrm{in} & \Omega\\
u_{n}=0 & \mathrm{on} & \mathbb{R}^{N}\setminus\Omega.
\end{array}
\right.
\]

Many of the results presented in Section \ref{Sec2} are contained in
\cite{BBMP} (for $p=2$) and \cite{CMS} (for $p>1$), but we will contribute
with some additional information. For example, we will prove that $\left[
u_{n}\right]  _{s,p}\leq\left[  u_{n+1}\right]  _{s,p}$ for all $n\in
\mathbb{N}.$ This property makes simpler the proof that $u_{n}$ converges
strongly to a solution of (\ref{palfa}) when $\left\{  u_{n}\right\}
_{n\in\mathbb{N}}$ is bounded in $W_{0}^{s,p}(\Omega).$ It also holds true for
the local version of the problem.

Our main results, related to the Sobolev inequalities (\ref{Sobolev1}) and
(\ref{Sobolev2}), will be proved in the Sections \ref{Sec3} and \ref{Sec4}, respectively.

\section{The fractional singular problem \label{Sec2}}

In this section we will provide a framework for the fractional singular
problem (\ref{palfa}). First, we will present results of uniqueness and
boundedness for the singular problem (\ref{palfa}). In the sequence we will
study a family of nonsingular problems whose solutions approach the solution
of (\ref{palfa}) when it exists. At last, we will present a result of
existence for (\ref{palfa}) in the case $0<\alpha\leq1.$

\subsection{Preliminaries}

Let us first fix the notation that will be used in the whole paper.

The duality pairing corresponding to the fractional $p$-Laplacian is defined
as
\begin{equation}
\left\langle (-\Delta_{p})^{s}\,u,v\right\rangle :=\int\int_{\mathbb{R}^{2N}%
}\frac{\left\vert u(x)-u(y)\right\vert ^{p-2}(u(x)-u(y))(v(x)-v(y))}%
{\left\vert x-y\right\vert ^{N+sp}}\mathrm{d}x\mathrm{d}y, \label{weak}%
\end{equation}
where $u,v\in W_{0}^{s,p}(\Omega).$ For the sake of clarity we will use the
following notation%
\begin{equation}
\widetilde{v}(x,y)=v(x)-v(y), \label{til}%
\end{equation}
which allows us to write
\[
\left\langle (-\Delta_{p})^{s}\,u,v\right\rangle =\int\int_{\mathbb{R}^{2N}%
}\frac{\left\vert \widetilde{u}(x,y)\right\vert ^{p-2}\widetilde
{u}(x,y)\widetilde{v}(x,y)}{\left\vert x-y\right\vert ^{N+sp}}\mathrm{d}%
x\mathrm{d}y
\]
and
\[
\left\langle (-\Delta_{p})^{s}\,u_{2}-(-\Delta_{p})^{s}\,u_{1},u_{2}%
-u_{1}\right\rangle =\int\int_{\mathbb{R}^{2N}}\frac{\left\vert \widetilde
{u_{2}}\right\vert ^{p-2}\widetilde{u_{2}}-\left\vert \widetilde{u_{1}%
}\right\vert ^{p-2}\widetilde{u_{1}}}{\left\vert x-y\right\vert ^{N+sp}%
}(\widetilde{u_{2}}-\widetilde{u_{1}})\mathrm{d}x\mathrm{d}y.
\]

We will adopt the standard notations $v_{+}$ and $r^{\prime}$ for,
respectively, the positive part of a function $v$ and the H\"{o}lder conjugate
of a number $r>1.$ Thus,
\[
v_{+}:=\max\left\{  v,0\right\}  \quad\mathrm{and}\quad r^{\prime}:=\frac
{r}{r-1}.
\]

\begin{remark}
\label{Sign}If a function $u\in W_{0}^{s,p}(\Omega)$ changes sign in $\Omega$
then $\left[  \left\vert u\right\vert \right]  _{s,p}^{p}<\left[  u\right]
_{s,p}^{p}.$ This stems from the following fact%
\[
\left\vert \left\vert u(x)\right\vert -\left\vert u(y)\right\vert \right\vert
<\left\vert u(x)-u(y)\right\vert \quad\mathrm{whenever}\quad u(x)u(y)<0.
\]

\end{remark}

The symbol $S_{\theta}$ will denote, for each $\theta\in\lbrack1,p_{s}^{\star
}),$ a positive constant satisfying
\begin{equation}
\left\Vert u\right\Vert _{\theta}^{p}\leq S_{\theta}\left[  u\right]
_{s,p}^{p},\quad\forall\,u\in W_{0}^{s,p}(\Omega). \label{Sq}%
\end{equation}
The existence of such a constant comes from the continuity of the embedding
$W_{0}^{s,p}(\Omega)\hookrightarrow L^{\theta}(\Omega).$ Accordingly, the
symbol $S_{p_{s}^{\star}}$ will be used to denote the constant relative to the
combined case $r=p_{s}^{\star}$ and $N>sp,$ since the embedding $W_{0}%
^{s,p}(\Omega)\hookrightarrow L^{p_{s}^{\star}}(\Omega)$ is also continuous in
this case.

\begin{definition}
\label{weaksol}We say that $u\in W_{0}^{s,p}(\Omega)$ is a weak solution of
the singular, fractional Dirichlet problem (\ref{palfa}), with $\alpha>0,$ if
the following conditions are satisfied:

\begin{enumerate}
\item[(i)] for each subdomain $\Omega^{\prime}$ compactly contained in
$\Omega$ there exists a positive constant $C_{\Omega^{\prime}}$ such that
\[
u\geq C_{\Omega^{\prime}}\quad\mathrm{a.e.}\text{ }\mathrm{in}\text{ }%
\Omega^{\prime}%
\]

\item[(ii)] for each $\varphi\in C_{c}^{\infty}(\Omega),$ one has
\end{enumerate}
\end{definition}

\begin{equation}
\left\langle (-\Delta_{p})^{s}\,u,\varphi\right\rangle =%
{\displaystyle\int_{\Omega}}
\frac{\omega\varphi}{u^{\alpha}}\mathrm{d}x. \label{distr}%
\end{equation}

Condition \textit{(i)}\textrm{ }arises from the singular nature of
(\ref{palfa}) and guarantees that the right-hand term of (\ref{distr}) is well
defined. The following proposition shows that the distributional formulation
\textit{(ii)} leads to the traditional notion of weak solution, according to
which the set of testing functions is taken to be $W_{0}^{s,p}(\Omega).$

\begin{proposition}
\label{wtest}Let $u\in W_{0}^{s,p}(\Omega)$ be a weak solution as defined
above. Then
\[
\left\langle (-\Delta_{p})^{s}\,u,\varphi\right\rangle =%
{\displaystyle\int_{\Omega}}
\frac{\omega\varphi}{u^{\alpha}}\mathrm{d}x,\quad\forall\,\varphi\in
W_{0}^{s,p}(\Omega).
\]

\end{proposition}

\begin{proof}
First we show, by using Fatou's Lemma and H\"{o}lder inequality, that
\begin{equation}
\left\vert
{\displaystyle\int_{\Omega}}
\dfrac{\omega v}{u^{\alpha}}\mathrm{d}x\right\vert \leq\left[  u\right]
_{s,p}^{p-1}\left[  v\right]  _{s,p},\quad\forall\,v\in W_{0}^{s,p}(\Omega).
\label{aux11}%
\end{equation}

Let $v$ be an arbitrary function in $W_{0}^{s,p}(\Omega)$ and take $\left\{
\xi_{n}\right\}  _{n\in\mathbb{N}}\subset C_{c}^{\infty}(\Omega)$ such that
$0\leq\xi_{n}\rightarrow\left\vert v\right\vert $ in $W_{0}^{s,p}(\Omega)$ and
also pointwise almost everywhere. Thus,
\begin{align*}
\left\vert
{\displaystyle\int_{\Omega}}
\dfrac{\omega v}{u^{\alpha}}\mathrm{d}x\right\vert  &  \leq%
{\displaystyle\int_{\Omega}}
\dfrac{\omega\left\vert v\right\vert }{u^{\alpha}}\mathrm{d}x\\
&  \leq\liminf_{n\rightarrow\infty}%
{\displaystyle\int_{\Omega}}
\dfrac{\omega\xi_{n}}{u^{\alpha}}\mathrm{d}x\\
&  =\liminf_{n\rightarrow\infty}\left\langle (-\Delta_{p})^{s}\,u,\xi
_{n}\right\rangle \\
&  \leq\left[  u\right]  _{s,p}^{p-1}\lim_{n\rightarrow\infty}\left[  \xi
_{n}\right]  _{s,p}=\left[  u\right]  _{s,p}^{p-1}\left[  \left\vert
v\right\vert \right]  _{s,p}\leq\left[  u\right]  _{s,p}^{p-1}\left[
v\right]  _{s,p}.
\end{align*}

Now, we fix $\varphi\in W_{0}^{s,p}(\Omega)$ and $\left\{  \varphi
_{n}\right\}  _{n\in\mathbb{N}}\subset C_{c}^{\infty}(\Omega)$ such that
$\varphi_{n}\rightarrow\varphi$ in $W_{0}^{s,p}(\Omega).$ Then, by taking
$v=\varphi_{n}-\varphi$ in (\ref{aux11}) we obtain%
\[
\lim_{n\rightarrow\infty}\left\vert
{\displaystyle\int_{\Omega}}
\dfrac{\omega\left(  \varphi_{n}-\varphi\right)  }{u^{\alpha}}\mathrm{d}%
x\right\vert \leq\lim_{n\rightarrow\infty}\left[  u\right]  _{s,p}%
^{p-1}\left[  \varphi_{n}-\varphi\right]  _{s,p}=0,
\]
that is,%
\[
\lim_{n\rightarrow\infty}%
{\displaystyle\int_{\Omega}}
\dfrac{\omega\varphi_{n}}{u^{\alpha}}\mathrm{d}x=%
{\displaystyle\int_{\Omega}}
\dfrac{\omega\varphi}{u^{\alpha}}\mathrm{d}x.
\]

Combining this fact with the strong convergence $\varphi_{n}\rightarrow
\varphi$ we can make $n\rightarrow\infty$ in the inequality
\[
\left\langle (-\Delta_{p})^{s}\,u,\varphi_{n}\right\rangle =%
{\displaystyle\int_{\Omega}}
\frac{\omega\varphi_{n}}{u^{\alpha}}\mathrm{d}x
\]
(recall that $u$ is a distributional solution), in order to obtain%
\[
\left\langle (-\Delta_{p})^{s}\,u,\varphi\right\rangle =%
{\displaystyle\int_{\Omega}}
\frac{\omega\varphi}{u^{\alpha}}\mathrm{d}x.
\]

\end{proof}

\subsection{Uniqueness}

The following lemma is well-known.

\begin{lemma}
\label{proputil} Let $p>1$ and $X,Y\in\mathbb{R}^{N}\setminus\{0\},$ $N\geq1.$
There exist positive constants $c_{p}$ and $C_{p},$ depending only on $p,$
such that
\end{lemma}

\begin{equation}
\left\vert \left\vert X\right\vert ^{p-2}X-\left\vert Y\right\vert
^{p-2}Y\right\vert \leq c_{p}\left\{
\begin{array}
[c]{ccl}%
\left\vert X-Y\right\vert ^{p-1} & \mathrm{if} & 1<p<2\\
(\left\vert X\right\vert +\left\vert Y\right\vert )^{p-2}\left\vert
X-Y\right\vert  & \mathrm{if} & p\geq2
\end{array}
\right.  \label{dinf1}%
\end{equation}
and
\begin{equation}
(\left\vert X\right\vert ^{p-2}X-\left\vert Y\right\vert ^{p-2}Y))(X-Y)\geq
C_{p}\left\{
\begin{array}
[c]{ccl}%
\dfrac{\left\vert X-Y\right\vert ^{2}}{(\left\vert X\right\vert +\left\vert
Y\right\vert )^{2-p}} & \mathrm{if} & 1<p<2\\
\left\vert X-Y\right\vert ^{p} & \mathrm{if} & p\geq2.
\end{array}
\right.  \label{dinf2}%
\end{equation}

\begin{lemma}
\label{C+}Let $v_{1},v_{2}\in W_{0}^{s,p}(\Omega)\setminus\{0\}.$ There exists
a positive constant $C,$ depending at most on $\Omega,N,s$ and $p,$ such that%
\[
\left\langle (-\Delta_{p})^{s}\,v_{1}-(-\Delta_{p})^{s}\,v_{2},v_{1}%
-v_{2}\right\rangle \geq C\left\{
\begin{array}
[c]{ccc}%
\dfrac{\left[  v_{1}-v_{2}\right]  _{s,p}^{2}}{\left(  \left[  v_{1}\right]
_{s,p}^{p}+\left[  v_{2}\right]  _{s,p}^{p}\right)  ^{\frac{2-p}{p}}} &
\mathrm{if} & 1<p<2\\
&  & \\
\left[  v_{1}-v_{2}\right]  _{s,p}^{p} & \mathrm{if} & p\geq2.
\end{array}
\right.
\]

\end{lemma}

\begin{proof}
When $p\geq2$ estimates (\ref{dinf2}) and (\ref{Sq}) yield
\begin{align*}
\left\langle (-\Delta_{p})^{s}\,v_{1}-(-\Delta_{p})^{s}\,v_{2},v_{1}%
-v_{2}\right\rangle  &  \geq C_{p}\int\int_{\mathbb{R}^{2N}}\dfrac{\left\vert
\widetilde{v_{1}}-\widetilde{v_{2}}\right\vert ^{p}}{\left\vert x-y\right\vert
^{N+sp}}\mathrm{d}x\mathrm{d}y\\
&  =C_{p}\left[  v_{1}-v_{2}\right]  _{s,p}^{p}.
\end{align*}

Now, let us consider the case $1<p<2.$ It follows from (\ref{dinf2}) that
\begin{align*}
\left\langle (-\Delta_{p})^{s}\,v_{1}-(-\Delta_{p})^{s}\,v_{2},v_{1}%
-v_{2}\right\rangle  &  =\int\int_{\mathbb{R}^{2N}}\frac{\left\vert
\widetilde{v_{1}}\right\vert ^{p-2}\widetilde{v_{1}}-\left\vert \widetilde
{v_{2}}\right\vert ^{p-2}\widetilde{v_{2}}}{\left\vert x-y\right\vert ^{N+sp}%
}(\widetilde{v_{1}}-\widetilde{v_{2}})\mathrm{d}x\mathrm{d}y\\
&  \geq C_{p}\int\int_{\mathbb{R}^{2N}}\dfrac{\left\vert \widetilde{v_{1}%
}-\widetilde{v_{2}}\right\vert ^{2}}{(\left\vert \widetilde{v_{1}}\right\vert
+\left\vert \widetilde{v_{2}}\right\vert )^{2-p}\left\vert x-y\right\vert
^{N+sp}}\mathrm{d}x\mathrm{d}y.
\end{align*}
H\"{o}lder inequality yields%
\begin{align*}
\left[  v_{1}-v_{2}\right]  _{s,p}^{p}  &  =\int\int_{\mathbb{R}^{2N}}%
\dfrac{\left\vert \widetilde{v_{1}}-\widetilde{v_{2}}\right\vert ^{p}%
}{\left\vert x-y\right\vert ^{N+sp}}\mathrm{d}x\mathrm{d}y\\
&  =\int\int_{\mathbb{R}^{2N}}\dfrac{\left\vert \widetilde{v_{1}}%
-\widetilde{v_{2}}\right\vert ^{p}(\left\vert \widetilde{v_{1}}\right\vert
+\left\vert \widetilde{v_{2}}\right\vert )^{\frac{p(2-p)}{2}}}{(\left\vert
\widetilde{v_{1}}\right\vert +\left\vert \widetilde{v_{2}}\right\vert
)^{\frac{p(2-p)}{2}}\left\vert x-y\right\vert ^{N+sp}}\mathrm{d}%
x\mathrm{d}y\leq A^{\frac{p}{2}}B^{\frac{2-p}{2}}%
\end{align*}
where%
\begin{align*}
A  &  =\int\int_{\mathbb{R}^{2N}}(\left\vert \widetilde{v_{1}}-\widetilde
{v_{2}}\right\vert ^{p}(\left\vert \widetilde{v_{1}}\right\vert +\left\vert
\widetilde{v_{2}}\right\vert )^{-\frac{p(2-p)}{2}}\left\vert x-y\right\vert
^{-(N+sp)\frac{p}{2}})^{\frac{2}{p}}\mathrm{d}x\mathrm{d}y\\
&  =\int\int_{\mathbb{R}^{2N}}\dfrac{\left\vert \widetilde{v_{1}}%
-\widetilde{v_{2}}\right\vert ^{2}}{(\left\vert \widetilde{v_{1}}\right\vert
+\left\vert \widetilde{v_{2}}\right\vert )^{2-p}\left\vert x-y\right\vert
^{N+sp}}\mathrm{d}x\mathrm{d}y
\end{align*}
and%
\begin{align*}
B  &  =\int\int_{\mathbb{R}^{2N}}\left(  (\widetilde{v_{1}}+\left\vert
\widetilde{v_{2}}\right\vert )^{\frac{p(2-p)}{2}}\left\vert x-y\right\vert
^{-(N+sp)\frac{2-p}{2}}\right)  ^{\frac{2}{2-p}}\mathrm{d}x\mathrm{d}y\\
&  =\int\int_{\mathbb{R}^{2N}}\dfrac{(\left\vert \widetilde{v_{1}}\right\vert
+\left\vert \widetilde{v_{2}}\right\vert )^{p}}{\left\vert x-y\right\vert
^{N+sp}}\mathrm{d}x\mathrm{d}y\\
&  \leq2^{p}\int\int_{\mathbb{R}^{2N}}\dfrac{\left\vert \widetilde{v_{1}%
}\right\vert ^{p}+\left\vert \widetilde{v_{2}}\right\vert ^{p}}{\left\vert
x-y\right\vert ^{N+sp}}\mathrm{d}x\mathrm{d}y=2^{p}\left(  \left[
v_{1}\right]  _{s,p}^{p}+\left[  v_{2}\right]  _{s,p}^{p}\right)  .
\end{align*}
Therefore,
\begin{align*}
\left\langle (-\Delta_{p})^{s}\,v_{1}-(-\Delta_{p})^{s}\,v_{2},v_{1}%
-v_{2}\right\rangle  &  \geq C_{p}A\\
&  \geq C_{p}\left(  \left[  v_{1}-v_{2}\right]  _{s,p}^{p}B^{-\frac{2-p}{2}%
}\right)  ^{\frac{2}{p}}\\
&  \geq C_{p}\left[  v_{1}-v_{2}\right]  _{s,p}^{2}\left(  2^{p}\left(
\left[  v_{1}\right]  _{s,p}^{p}+\left[  v_{2}\right]  _{s,p}^{p}\right)
\right)  ^{-\frac{2-p}{p}}\\
&  =\frac{C\left[  v_{1}-v_{2}\right]  _{s,p}^{2}}{\left(  \left[
v_{1}\right]  _{s,p}^{p}+\left[  v_{2}\right]  _{s,p}^{p}\right)  ^{\frac
{2-p}{p}}}.
\end{align*}

\end{proof}

At this point we can already prove that weak solutions are unique.

\begin{theorem}
\label{uniqueness}The singular fractional Dirichlet problem (\ref{palfa}),
with $\alpha>0,$ has at most one weak solution.
\end{theorem}

\begin{proof}
Let us suppose that $u,v\in W_{0}^{s,p}(\Omega)$ are weak solutions of
(\ref{palfa}). Then, according to Proposition \ref{wtest}, we have%
\[
\left\langle (-\Delta_{p})^{s}\,u-(-\Delta_{p})^{s}\,v,u-v\right\rangle =%
{\displaystyle\int_{\Omega}}
\omega\left(  \frac{1}{u^{\alpha}}-\frac{1}{v^{\alpha}}\right)
(u-v)\mathrm{d}x\leq0,
\]
since the integrand of the right-hand term is not positive in $\Omega.$ Thus,
according to Lemma \ref{C+}, we must have $\left[  u-v\right]  _{s,p}=0,$
showing that $u=v$ almost everywhere.
\end{proof}

\subsection{$L^{\infty}$ bounds\label{secbounds}}

The following lemma can be found in \cite[Lemma 2.1]{stam}. For the sake of
completeness, we sketch its proof.

\begin{lemma}
\label{Stamlem}Let $g$ be a nonnegative and nonincreasing function defined for
all $t\geq k_{0}$ and such that
\begin{equation}
g(h)\leq\frac{C}{(h-k)^{\theta}}[g(k)]^{b},\quad\mathrm{whenever}\quad
k_{0}\leq k<h, \label{stam1}%
\end{equation}
where $C,\theta$ and $b$ are constants, $C,\theta>0$ and $b>1$. Then,
\begin{equation}
g(k_{0}+d)=0, \label{stam3}%
\end{equation}
where $d^{\theta}=C[g(k_{0})]^{b-1}2^{\theta b/(b-1)}.$
\end{lemma}

\begin{proof}
Let $\left\{  k_{n}\right\}  _{n\in\mathbb{N}}$ be the increasing sequence
defined by $k_{n}:=k_{0}+d-\dfrac{d}{2^{n}}<k_{0}+d.$ Using (\ref{stam1}) one
can show, by induction, that%
\[
g(k_{n})\leq g(k_{0})2^{-\frac{na}{b-1}}.
\]
Hence, since $0\leq g(k_{0}+d)\leq g(k_{n})$ we obtain (\ref{stam3}), after
making $n\rightarrow\infty.$
\end{proof}

\begin{theorem}
\label{global}Let $\alpha>0$ and $\omega\in L^{r}(\Omega),$ with $pr^{\prime
}<p_{s}^{\star}.$ If $u\in W_{0}^{s,p}(\Omega)$ is positive in $\Omega$ and
satisfies%
\[
\left\langle (-\Delta_{p})^{s}\,u,\varphi\right\rangle \leq%
{\displaystyle\int_{\Omega}}
\dfrac{\omega\varphi}{u^{\alpha}}\mathrm{d}x,\quad\forall\,\varphi\in
W_{0}^{s,p}(\Omega),\quad\varphi\geq0,
\]
then $u\in L^{\infty}(\Omega).$ Moreover, for each $pr^{\prime}<\theta
<p_{s}^{\star},$ one has%
\begin{equation}
\left\Vert u\right\Vert _{\infty}\leq C_{\alpha}\left(  \frac{\left\Vert
\omega\right\Vert _{r}}{S_{\theta}}\right)  ^{\frac{1}{p-1+\alpha}}%
2^{\frac{b(p-1)}{(b-1)(p-1+\alpha)}}\left\vert \Omega\right\vert
^{\frac{(b-1)(p-1)}{\theta(p-1+\alpha)}} \label{lbound}%
\end{equation}
where
\begin{equation}
C_{\alpha}:=\left(  \frac{\alpha}{p-1}\right)  ^{\frac{p-1}{p-1+\alpha}%
}\left(  1+\frac{p-1}{\alpha}\right)  \quad\mathrm{and}\quad b:=(\frac{\theta
}{r^{\prime}}-1)\frac{1}{p-1}>1. \label{Cab}%
\end{equation}

\end{theorem}

\begin{proof}
Let
\[
A_{k}:=\{x\in\Omega:u(x)>k\}
\]
be the $k$-super-level set of $u,$ for each $k\geq0.$ Since $(u-k)_{+}\in
W_{0}^{s,p}(\Omega)$ we obtain
\begin{align*}
\left[  (u-k)_{+}\right]  _{s,p}^{p}  &  \leq\left\langle \left(  -\Delta
_{p}\right)  ^{s}\,u,(u-k)_{+}\right\rangle \\
&  \leq\int_{\Omega}\frac{\omega}{u^{\alpha}}(u-k)_{+}\mathrm{d}x\\
&  =\int_{A_{k}}\frac{\omega}{u^{\alpha}}(u-k)\mathrm{d}x\leq\frac{\left\Vert
\omega\right\Vert _{r}}{k^{\alpha}}\left(  \int_{A_{k}}(u-k)^{r^{\prime}%
}\mathrm{d}x\right)  ^{\frac{1}{r^{\prime}}},
\end{align*}
where the first inequality can be easily checked.

Let $\theta$ be such that $pr^{\prime}<\theta<p_{s}^{\star}.$ Then, the
continuity of the Sobolev embedding $W_{0}^{s,p}(\Omega)\hookrightarrow
L^{\theta}(\Omega)$ and the H\"{o}lder inequality imply that
\begin{align*}
S_{\theta}\left(  \int_{A_{k}}(u-k)^{\theta}\mathrm{d}x\right)  ^{\frac
{p}{\theta}}  &  =S_{\theta}\left(  \int_{\Omega}(u-k)_{+}^{\theta}%
\mathrm{d}x\right)  ^{\frac{p}{\theta}}\\
&  \leq\left[  (u-k)_{+}\right]  _{s,p}^{p}\\
&  \leq\frac{\left\Vert \omega\right\Vert _{r}}{k^{\alpha}}\left(  \int
_{A_{k}}(u-k)^{r^{\prime}}\mathrm{d}x\right)  ^{\frac{1}{r^{\prime}}}\\
&  \leq\frac{\left\Vert \omega\right\Vert _{r}}{k^{\alpha}}\left(  \int
_{A_{k}}(u-k)^{\theta}\mathrm{d}x\right)  ^{\frac{1}{\theta}}\left\vert
A_{k}\right\vert ^{\frac{1}{r^{\prime}}-\frac{1}{\theta}}%
\end{align*}
so that%
\begin{equation}
S_{\theta}\left(  \int_{A_{k}}(u-k)^{\theta}\mathrm{d}x\right)  ^{\frac
{p-1}{\theta}}\leq\frac{\left\Vert \omega\right\Vert _{r}}{k^{\alpha}%
}\left\vert A_{k}\right\vert ^{\frac{1}{r^{\prime}}-\frac{1}{\theta}}.
\label{stam4}%
\end{equation}

Let $0<k_{0}\leq k<h.$ Then, $A_{h}\subset A_{k}$ and
\begin{equation}
\left\vert A_{h}\right\vert ^{\frac{1}{\theta}}(h-k)=\left(  \int_{A_{h}%
}(h-k)^{\theta}\mathrm{d}x\right)  ^{\frac{1}{\theta}}\leq\left(  \int_{A_{h}%
}(u-k)^{\theta}\mathrm{d}x\right)  ^{\frac{1}{\theta}}\leq\left(  \int_{A_{k}%
}(u-k)^{\theta}\mathrm{d}x\right)  ^{\frac{1}{\theta}}.\nonumber
\end{equation}
After combining this with (\ref{stam4}) we get (recall that $k^{\alpha}%
\geq(k_{0})^{\alpha}$)%
\[
S_{\theta}\left\vert A_{h}\right\vert ^{\frac{p-1}{\theta}}(h-k)^{p-1}%
\leq\frac{\left\Vert \omega\right\Vert _{r}}{(k_{0})^{\alpha}}\left\vert
A_{k}\right\vert ^{\frac{1}{r^{\prime}}-\frac{1}{\theta}}%
\]
which can be rewritten as%
\[
g(h)\leq\frac{C}{(h-k)^{\theta}}[g(k)]^{b}%
\]
where%
\[
g(t):=\left\vert A_{t}\right\vert ,\quad C:=\left(  \frac{\left\Vert
\omega\right\Vert _{r}}{S_{\theta}(k_{0})^{\alpha}}\right)  ^{\frac{\theta
}{p-1}}%
\]
and
\[
b=(\frac{1}{r^{\prime}}-\frac{1}{\theta})\frac{\theta}{p-1}=(\frac{\theta
}{r^{\prime}}-1)\frac{1}{p-1}>(p-1)\frac{1}{p-1}=1.
\]

It follows from Lemma \ref{Stamlem}, with $d^{\theta}=C[\left\vert
A_{1}\right\vert ]^{b-1}2^{\theta b/(b-1)},$ that%
\[
\left\vert A_{t}\right\vert \leq\left\vert A_{k_{0}+d}\right\vert
=0,\quad\forall\,t\geq k_{0}+d.
\]
This fact shows that $u\in L^{\infty}(\Omega)$ and
\[
\left\Vert u\right\Vert _{\infty}\leq k_{0}+d\leq k_{0}+\left(  \frac
{\left\Vert \omega\right\Vert _{r}}{S_{\theta}(k_{0})^{\alpha}}\right)
^{\frac{1}{p-1}}2^{b/(b-1)}\left\vert \Omega\right\vert ^{\frac{b-1}{\theta}%
}=k_{0}+(k_{0})^{-\frac{\alpha}{p-1}}A
\]
where
\[
A=\left(  \frac{\left\Vert \omega\right\Vert _{r}}{S_{\theta}}\right)
^{\frac{1}{p-1}}2^{b/(b-1)}\left\vert \Omega\right\vert ^{\frac{b-1}{\theta}%
}.
\]

After choosing the optimal value of $k_{0}$ we obtain%
\[
\left\Vert u\right\Vert _{\infty}\leq\left(  \frac{\alpha}{p-1}\right)
^{\frac{p-1}{p-1+\alpha}}\left(  1+\frac{p-1}{\alpha}\right)  \left(
\frac{\left\Vert \omega\right\Vert _{r}}{S_{\theta}}\right)  ^{\frac
{1}{p-1+\alpha}}2^{\frac{b(p-1)}{(b-1)(p-1+\alpha)}}\left\vert \Omega
\right\vert ^{\frac{(b-1)(p-1)}{\theta(p-1+\alpha)}}.
\]

\end{proof}

\begin{remark}
\label{Global1}When $sp<N$ the proof of Theorem \ref{global} applies if
$r>\dfrac{N}{sp}$ and $\theta=p_{s}^{\star}.$ In this case, the estimate
(\ref{lbound}) becomes
\[
\left\Vert u\right\Vert _{\infty}\leq C_{\alpha}\left(  \frac{\left\Vert
\omega\right\Vert _{r}}{S_{p_{s}^{\star}}}\right)  ^{\frac{1}{p-1+\alpha}%
}2^{\frac{p_{s}^{\star}-r^{\prime}}{p_{s}^{\star}-pr^{\prime}}\frac
{p-1}{p-1+\alpha}}\left\vert \Omega\right\vert ^{\frac{p_{s}^{\star}%
-r^{\prime}p}{r^{\prime}p_{s}^{\star}}\frac{p-1}{p-1+\alpha}}.
\]

When $sp\geq N$ the condition $pr^{\prime}<p_{s}^{\star}=\infty$ naturally
holds true if $r>1,$ in which case the estimate (\ref{lbound}) is valid for
any fixed $\theta>pr^{\prime}.$
\end{remark}

\subsection{A family of approximating problems}

The following lemma is inspired by the proof of Lemma 9 of \cite{LindLind}.

\begin{lemma}
\label{lemma+}Let $v_{1},v_{2}\in W_{0}^{s,p}(\Omega)$ and denote
$v=v_{1}-v_{2}.$ Then,%
\[
\left\langle (-\Delta_{p})^{s}\,v_{1}-(-\Delta_{p})^{s}\,v_{2},v_{+}%
\right\rangle \geq(p-1)\int\int_{\mathbb{R}^{2N}}\dfrac{\left\vert
v_{+}(x)-v_{+}(y)\right\vert ^{2}Q(x,y)}{\left\vert x-y\right\vert ^{N+sp}%
}\mathrm{d}x\mathrm{d}y,
\]
where%
\begin{equation}
Q(x,y)=\int_{0}^{1}\left\vert \widetilde{v_{2}}(x,y)+t(\widetilde{v_{1}%
}(x,y)-\widetilde{v_{2}}(x,y))\right\vert ^{p-2}dt\geq0. \label{Q}%
\end{equation}

\end{lemma}

\begin{proof}
Making use of the identity
\[
\left\vert b\right\vert ^{p-2}b-\left\vert a\right\vert ^{p-2}a=(p-1)(b-a)\int
_{0}^{1}\left\vert a+t(b-a)\right\vert ^{p-2}dt
\]
we obtain%
\[
\left\vert \widetilde{v_{1}}(x,y)\right\vert ^{p-2}\widetilde{v_{1}%
}(x,y)-\left\vert \widetilde{v_{2}}(x,y)\right\vert ^{p-2}\widetilde{v_{2}%
}(x,y)=(p-1)\left(  \widetilde{v_{1}}(x,y)-\widetilde{v_{2}}(x,y)\right)
Q(x,y)
\]
where $Q$ is given by (\ref{Q}).

Hence, we can write (recall that $v=v_{1}-v_{2}$)
\[
\left\langle (-\Delta_{p})^{s}\,v_{1}-(-\Delta_{p})^{s}v_{2},v_{+}%
\right\rangle =(p-1)\int\int_{\mathbb{R}^{2N}}\dfrac{\left(  \widetilde{v_{1}%
}(x,y)-\widetilde{v_{2}}(x,y)\right)  \,\widetilde{v_{+}}(x,y)\,Q(x,y)}%
{\left\vert x-y\right\vert ^{N+sp}}\mathrm{d}x\mathrm{d}y.
\]

Since%
\begin{align*}
\left(  \widetilde{v_{1}}(x,y)-\widetilde{v_{2}}(x,y)\right)  \widetilde
{v_{+}}(x,y)  &  =\left(  v_{1}(x)-v_{1}(y)-v_{2}(x)+v_{2}(y)\right)  \left(
v_{+}(x)-v_{+}(y)\right) \\
&  =\left(  v(x)-v(y)\right)  \left(  v_{+}(x)-v_{+}(y)\right)  \geq\left\vert
v_{+}(x)-v_{+}(y)\right\vert ^{2}%
\end{align*}
the proof is complete. (The latter inequality is very simple to check.)
\end{proof}

In the sequel we will show that, for each $n\in\mathbb{N},$ there exists a
unique function $u_{n}\in W_{0}^{s,p}(\Omega)\cap L^{\infty}(\Omega)$ such
that
\begin{equation}
\left\{
\begin{array}
[c]{lcl}%
(-\Delta_{p})^{s}\,u_{n}=\dfrac{\omega_{n}}{(u_{n}+\frac{1}{n})^{\alpha}} &
\mathrm{in} & \Omega\\
u_{n}>0 & \mathrm{in} & \Omega\\
u_{n}=0 & \mathrm{on} & \mathbb{R}^{N}\setminus\Omega,
\end{array}
\right.  \label{Pn}%
\end{equation}
where%
\[
\alpha>0\quad\mathrm{and}\quad\omega_{n}(x):=\min\left\{  \omega(x),n\right\}
.
\]

\begin{proposition}
\label{un}Let $\alpha>0$ and $\omega\in L^{1}(\Omega)\diagdown\left\{
0\right\}  ,\;\omega\geq0.$ For each $n\in\mathbb{N}$ there exists a unique
function $u_{n}\in W_{0}^{s,p}(\Omega)$ satisfying (\ref{Pn}) in the weak
sense, that is,%
\begin{equation}
\left\langle (-\Delta_{p})^{s}\,u_{n},\varphi\right\rangle =%
{\displaystyle\int_{\Omega}}
\dfrac{\omega_{n}\varphi}{(u_{n}+\frac{1}{n})^{\alpha}}\mathrm{d}%
x,\quad\forall\,\varphi\in W_{0}^{s,p}(\Omega). \label{weakun}%
\end{equation}
Moreover, $u_{n}$ is strictly positive in $\Omega,$ belongs to $C^{\beta_{s}%
}(\overline{\Omega}),$ for some $\beta_{s}\in(0,s]$ and%
\begin{equation}
\left[  u_{n}\right]  _{s,p}^{p}\leq\left[  \varphi\right]  _{s,p}^{p}+p%
{\displaystyle\int_{\Omega}}
\dfrac{\omega_{n}(u_{n}-\varphi)}{(u_{n}+\frac{1}{n})^{\alpha}}\mathrm{d}%
x,\quad\forall\,\varphi\in W_{0}^{s,p}(\Omega). \label{gmonot}%
\end{equation}

\end{proposition}

\begin{proof}
We will obtain $u_{n}$ as a fixed point of the operator $T:L^{p}%
(\Omega)\longrightarrow W_{0}^{s,p}(\Omega)\hookrightarrow L^{p}(\Omega)$ that
associates to each $w\in L^{p}(\Omega)$ the only weak solution $v=T(w)\in
W_{0}^{s,p}(\Omega)$ of the nonsingular Dirichlet problem%
\[
\left\{
\begin{array}
[c]{lcl}%
(-\Delta_{p})^{s}\,u=\dfrac{\omega_{n}}{(\left\vert w\right\vert +\frac{1}%
{n})^{\alpha}} & \mathrm{in} & \Omega\\
u=0 & \mathrm{on} & \mathbb{R}^{N}\setminus\Omega.
\end{array}
\right.
\]
The function $v$ is obtained through a direct minimization method applied to
the functional
\[
\varphi\in W_{0}^{s,p}(\Omega)\longmapsto\frac{1}{p}\left[  \varphi\right]
_{s,p}^{p}-%
{\displaystyle\int_{\Omega}}
\dfrac{\omega_{n}\varphi}{(\left\vert w\right\vert +\frac{1}{n})^{\alpha}%
}\mathrm{d}x,
\]
which is strictly convex and of class $C^{1}.$ Thus, $v$ is both the only
minimizer and the only critical point of this functional. Hence,%
\begin{equation}
\frac{1}{p}\left[  v\right]  _{s,p}^{p}-%
{\displaystyle\int_{\Omega}}
\dfrac{\omega_{n}v}{(\left\vert w\right\vert +\frac{1}{n})^{\alpha}}%
\mathrm{d}x\leq\frac{1}{p}\left[  \varphi\right]  _{s,p}^{p}-%
{\displaystyle\int_{\Omega}}
\dfrac{\omega_{n}\varphi}{(\left\vert w\right\vert +\frac{1}{n})^{\alpha}%
}\mathrm{d}x,\quad\forall\,\varphi\in W_{0}^{s,p}(\Omega) \label{aux6}%
\end{equation}
and%
\begin{equation}
\left\langle (-\Delta_{p})^{s}\,v,\varphi\right\rangle =%
{\displaystyle\int_{\Omega}}
\dfrac{\omega_{n}\varphi}{(\left\vert w\right\vert +\frac{1}{n})^{\alpha}%
}\mathrm{d}x,\quad\forall\,\varphi\in W_{0}^{s,p}(\Omega). \label{aux3}%
\end{equation}

It follows that%
\begin{equation}
\left[  v\right]  _{s,p}^{p}=\left\langle (-\Delta_{p})^{s}\,v,v\right\rangle
=%
{\displaystyle\int_{\Omega}}
\dfrac{\omega_{n}v}{(\left\vert w\right\vert +\frac{1}{n})^{\alpha}}%
\mathrm{d}x\leq n^{\alpha+1}\left\Vert v\right\Vert _{1}\leq S_{1}^{\frac
{1}{p}}n^{\alpha+1}\left[  v\right]  _{s,p}, \label{aux1}%
\end{equation}
where $S_{1}$ is a positive constant that is uniform with respect to $v$ (we
have used the continuity of the embedding $W_{0}^{s,p}(\Omega)\hookrightarrow
L^{1}(\Omega)$).

It follows from (\ref{aux1}) that%
\begin{equation}
\left[  T(w)\right]  _{s,p}\leq\left(  S_{1}^{\frac{1}{p}}n^{\alpha+1}\right)
^{\frac{1}{p-1}},\quad\forall\,w\in L^{p}(\Omega) \label{aux2}%
\end{equation}
and thus, by taking into account the compactness of the embedding $W_{0}%
^{s,p}(\Omega)\hookrightarrow L^{p}(\Omega),$ we conclude that the operator
$T:L^{p}(\Omega)\longrightarrow L^{p}(\Omega)$ is compact.

We are going to show, by contradiction, that $T$ is also continuous. Thus, we
assume that there exist $\epsilon>0$ and $w_{k}\rightarrow w$ in $L^{p}%
(\Omega)$ such that
\begin{equation}
\left\Vert v_{k}-v\right\Vert _{p}>\epsilon\quad\forall\,k\in\mathbb{N},
\label{contra}%
\end{equation}
where $v_{k}:=T(w_{k})$ and $v:=T(w).$ We can also assume, without loss of
generality, that $\left\vert w_{k}\right\vert \rightarrow\left\vert
w\right\vert $ almost everywhere in $\Omega$ (this comes from the convergence
in $L^{p}(\Omega)$).

It follows from (\ref{aux3}), with $\varphi=v_{k}-v,$ that%
\begin{align*}
\left\langle (-\Delta_{p})^{s}\,v_{k}-(-\Delta_{p})^{s}\,v,v_{k}%
-v\right\rangle  &  =%
{\displaystyle\int_{\Omega}}
(v_{k}-v)\left(  \dfrac{\omega_{n}}{(\left\vert w_{k}\right\vert +\frac{1}%
{n})^{\alpha}}-\dfrac{\omega_{n}}{(\left\vert w\right\vert +\frac{1}%
{n})^{\alpha}}\right)  \mathrm{d}x\\
&  =%
{\displaystyle\int_{\Omega}}
(v_{k}-v)\omega_{n}h_{k}\mathrm{d}x,
\end{align*}
where%
\[
h_{k}:=\dfrac{1}{(\left\vert w_{k}\right\vert +\frac{1}{n})^{\alpha}}%
-\dfrac{1}{(\left\vert w\right\vert +\frac{1}{n})^{\alpha}}.
\]

Therefore,
\begin{equation}
\left\vert \left\langle (-\Delta_{p})^{s}\,v_{k}-(-\Delta_{p})^{s}%
\,v,v_{k}-v\right\rangle \right\vert \leq n%
{\displaystyle\int_{\Omega}}
\left\vert v_{k}-v\right\vert \left\vert h_{k}\right\vert \mathrm{d}x\leq
n\left\Vert v_{k}-v\right\Vert _{p}\left\Vert h_{k}\right\Vert _{p^{\prime}}.
\label{aux4}%
\end{equation}
Since $\left\vert h_{k}\right\vert \leq n^{\alpha}$ and $\lim
\limits_{k\rightarrow\infty}\left\vert h_{k}\right\vert \rightarrow0$ almost
everywhere in $\Omega$, Dominated Convergence Theorem guarantees that%
\begin{equation}
\lim_{k\rightarrow\infty}\left\Vert h_{k}\right\Vert _{p^{\prime}}=0.
\label{aux5}%
\end{equation}

At this point we consider separately the cases $1<p<2$ and $p\geq2.$

\textbf{Case }$1<p<2.$ In this case, it follows from Lemma \ref{C+} and
(\ref{aux2}) that
\[
\left\langle (-\Delta_{p})^{s}\,v_{k}-(-\Delta_{p})^{s}\,v,v_{k}%
-v\right\rangle \geq\frac{C\left[  v_{k}-v\right]  _{s,p}^{2}}{\left(  \left[
v_{k}\right]  _{s,p}^{p}+\left[  v\right]  _{s,p}^{p}\right)  ^{\frac{2-p}{p}%
}}\geq\frac{CS_{p}^{\frac{2}{p}}\left\Vert v_{k}-v\right\Vert _{p}^{2}%
}{\left(  2\left(  S_{1}^{\frac{1}{p}}n^{\alpha+1}\right)  ^{\frac{p}{p-1}%
}\right)  ^{\frac{2-p}{p}}},
\]
that is,
\[
\left\langle (-\Delta_{p})^{s}\,v_{k}-(-\Delta_{p})^{s}\,v,v_{k}%
-w\right\rangle \geq C_{n}\left\Vert v_{k}-v\right\Vert _{p}^{2}%
\]
where the positive constant $C_{n}$ does not depend on $k.$

After combining this inequality with (\ref{aux4}) and (\ref{aux5}) we obtain%
\[
\lim_{k\rightarrow\infty}\left\Vert v_{k}-v\right\Vert _{p}\leq\frac{n}{C_{n}%
}\lim_{k\rightarrow\infty}\left\Vert h_{k}\right\Vert _{p^{\prime}}=0,
\]
which contradicts (\ref{contra}).

\textbf{Case }$p\geq2.$ In this case Lemma \ref{C+} and (\ref{aux4}) yield
\begin{align*}
CS_{p}\left\Vert v_{k}-v\right\Vert _{p}^{p}  &  \leq C\left[  v_{k}-v\right]
_{s,p}^{p}\\
&  \leq\left\langle (-\Delta_{p})^{s}\,v_{k}-(-\Delta_{p})^{s}\,v,v_{k}%
-v\right\rangle \leq n\left\Vert v_{k}-v\right\Vert _{p}\left\Vert
h_{k}\right\Vert _{p^{\prime}}.
\end{align*}
Hence, after using (\ref{aux5}) we arrive at
\[
\lim_{k\rightarrow\infty}\left\Vert v_{k}-v\right\Vert _{p}\leq\lim
_{k\rightarrow\infty}\left(  \frac{n}{CS_{p}}\left\Vert h_{k}\right\Vert
_{p^{\prime}}\right)  ^{\frac{1}{p-1}}=0,
\]
which also contradicts (\ref{contra}).

We have proved that $T:L^{p}(\Omega)\longrightarrow L^{p}(\Omega)$ is compact
and continuous. Moreover, (\ref{aux2}) implies that $T$ leaves invariant the
ball $\left\{  w\in L^{p}(\Omega):\left\Vert w\right\Vert _{p}\leq\left(
S_{1}n^{\alpha+1}\right)  ^{\frac{1}{p-1}}\right\}  .$ Therefore, by applying
Schauder's Fixed Point Theorem we conclude that $T$ has a fixed point $u_{n}$
in this ball. Of course,
\[
\left\{
\begin{array}
[c]{lcl}%
(-\Delta_{p})^{s}\,u_{n}=\dfrac{\omega_{n}}{(\left\vert u_{n}\right\vert
+\frac{1}{n})^{\alpha}} & \mathrm{in} & \Omega\\
u_{n}=0 & \mathrm{on} & \mathbb{R}^{N}\setminus\Omega
\end{array}
\right.
\]
in the weak sense.

Since the right-hand term of the above equation is nonnegative and belongs to
$L^{\infty}(\Omega)$, we can apply the comparison principle for the fractional
$p$-Laplacian (see \cite[Lemma 9]{LindLind}) and the main result of \cite{IMS}
to conclude, respectively, that $u_{n}$ is nonnegative and belongs to
$C^{\beta_{s}}(\overline{\Omega})$ for some $\beta_{s}\in(0,s]$ ($\beta_{s}$
does not depend neither on $\alpha$ nor on $n$).

It follows from \cite[Theorem A.1]{BRASCO FRANZINA} that $u_{n}>0$ almost
everywhere in $\Omega.$ Let us show, by employing a nonlocal Harnack
inequality proved in \cite{CKP}, that $u_{n}(x)>0$ for all $x\in\Omega.$
Suppose, by the way of contradiction, that $u_{n}(x_{0})=0$ for some $x_{0}%
\in\Omega.$ According Lemma 4.1 of \cite{CKP}, there exist positive constants
$\epsilon$ and $c$ (with $0<\epsilon<1\leq c$) such that
\[
\left(  \frac{1}{\left\vert B(x_{0})\right\vert }%
{\displaystyle\int_{B(x_{0})}}
(u_{n})^{\epsilon}\mathrm{d}x\right)  ^{\frac{1}{\epsilon}}\leq c\inf
_{B(x_{0})}u_{n}%
\]
where $B(x_{0})$ denotes a ball centered at $x_{0}$ and contained in $\Omega.$
Since, $\inf_{B(x_{0})}u_{n}=u_{n}(x_{0})=0$ the above inequality implies that
$u$ is identically null in $B(x_{0}),$ contradicting thus the fact that $u>0$
almost everywhere.

In order to prove the uniqueness of $u_{n}$ we assume that $v_{i},$
$i\in\left\{  1,2\right\}  ,$ satisfies%
\[
\left\{
\begin{array}
[c]{lcl}%
(-\Delta_{p})^{s}\,v_{i}=\dfrac{\omega_{n}}{(v_{i}+\frac{1}{n})^{\alpha}} &
\mathrm{in} & \Omega\\
v_{i}\geq0 & \mathrm{in} & \Omega\\
v_{i}=0 & \mathrm{on} & \mathbb{R}^{N}\setminus\Omega.
\end{array}
\right.
\]
Then,%
\[
\left\langle (-\Delta_{p})^{s}\,v_{1}-(-\Delta_{p})^{s}\,v_{2},v_{1}%
-v_{2}\right\rangle =%
{\displaystyle\int_{\Omega}}
(v_{1}-v_{2})\left(  \dfrac{\omega_{n}}{(v_{1}+\frac{1}{n})^{\alpha}}%
-\dfrac{\omega_{n}}{(v_{2}+\frac{1}{n})^{\alpha}}\right)  \mathrm{d}x\leq0,
\]
since the integrand of the right-hand term is not positive in $\Omega.$

On the other hand, by applying Lemma \ref{C+}, we conclude that $\left[
v_{1}-v_{2}\right]  _{s,p}=0,$ showing that $v_{1}=v_{2}$ almost everywhere in
$\Omega.$

We finish this proof by observing that (\ref{gmonot}) follows directly from
(\ref{aux6}), with $w=u_{n}$ and $v=T(u_{n})=u_{n}:$
\[
\frac{1}{p}\left[  u_{n}\right]  _{s,p}^{p}-%
{\displaystyle\int_{\Omega}}
\dfrac{\omega_{n}u_{n}}{(u_{n}+\frac{1}{n})^{\alpha}}\mathrm{d}x\leq\frac
{1}{p}\left[  \varphi\right]  _{s,p}^{p}-%
{\displaystyle\int_{\Omega}}
\dfrac{\omega_{n}\varphi}{(u_{n}+\frac{1}{n})^{\alpha}}\mathrm{d}%
x,\quad\forall\,\varphi\in W_{0}^{s,p}(\Omega).
\]

\end{proof}

\begin{proposition}
\label{unmonot}The sequences $\left\{  u_{n}\right\}  _{n\in\mathbb{N}}\subset
W_{0}^{s,p}(\Omega)$ and $\left\{  \left[  u_{n}\right]  _{s,p}\right\}
_{n\in\mathbb{N}}\subset(0,\infty)$ are nondecreasing, that is%
\[
u_{n}\leq u_{n+1}\quad\mathrm{in}\quad\Omega,\quad\mathrm{and}\quad\left[
u_{n}\right]  _{s,p}\leq\left[  u_{n+1}\right]  _{s,p},\quad\forall
\,n\in\mathbb{N}.
\]

\end{proposition}

\begin{proof}
Let $\varphi:=u_{n}-u_{n+1}.$ It follows from (\ref{weakun}) that%
\[
\left\langle (-\Delta_{p})^{s}\,u_{n}-(-\Delta_{p})^{s}\,u_{n+1},\varphi
_{+}\right\rangle =%
{\displaystyle\int_{\Omega}}
\dfrac{\omega_{n}\varphi_{+}}{(u_{n}+\frac{1}{n})^{\alpha}}-\dfrac
{\omega_{n+1}\varphi_{+}}{(u_{n+1}+\frac{1}{n+1})^{\alpha}}\mathrm{d}x.
\]

Since
\[
0\leq\omega_{n}(x)=\min\left\{  \omega(x),n\right\}  \leq\min\left\{
\omega(x),n+1\right\}  =\omega_{n+1}(x)
\]
we have $\omega_{n}\varphi_{+}\leq\omega_{n+1}\varphi_{+}$ and, hence,%
\begin{equation}
\left\langle (-\Delta_{p})^{s}\,u_{n}-(-\Delta_{p})^{s}\,u_{n+1},\varphi
_{+}\right\rangle \leq%
{\displaystyle\int_{\Omega}}
\left(  \dfrac{\omega_{n+1}\varphi_{+}}{(u_{n}+\frac{1}{n})^{\alpha}}%
-\dfrac{\omega_{n+1}\varphi_{+}}{(u_{n+1}+\frac{1}{n+1})^{\alpha}}\right)
\mathrm{d}x\leq0. \label{aux7}%
\end{equation}
since the integrand above is not positive.

On the other hand, it follows from Lemma \ref{lemma+} that
\begin{equation}
\left\langle (-\Delta_{p})^{s}\,u_{n}-(-\Delta_{p})^{s}\,u_{n+1},\varphi
_{+}\right\rangle \geq(p-1)\int\int_{\mathbb{R}^{2N}}\dfrac{\left\vert
\varphi_{+}(x)-\varphi_{+}(y)\right\vert ^{2}Q(x,y)}{\left\vert x-y\right\vert
^{N+sp}}\mathrm{d}x\mathrm{d}y\geq0, \label{aux8}%
\end{equation}
where%
\[
Q(x,y)=\int_{0}^{1}\left\vert \widetilde{u_{n+1}}(x,y)+t(\widetilde{u_{n}%
}(x,y)-\widetilde{u_{n+1}}(x,y))\right\vert ^{p-2}\mathrm{d}t\geq0.
\]
Note that $Q(x,y)=0$ implies that $\widetilde{u_{n+1}}(x,y)=\widetilde{u_{n}%
}(x,y)=0,$ a pair of equalities that lead to $\varphi(x)=\varphi(y).$

After comparing (\ref{aux8}) with (\ref{aux7}) we can conclude that%
\[
\left\vert \varphi_{+}(x)-\varphi_{+}(y)\right\vert ^{2}Q(x,y)=0
\]
at almost every point $(x,y)\in\mathbb{R}^{2N},$ implying that $\varphi
_{+}(x)=\varphi_{+}(y)$ at almost every point $(x,y).$ Since $\varphi$ is zero
out of $\Omega,$ this fact implies that $\varphi_{+}=0$ almost everywhere.
That is, $u_{n}-u_{n+1}\leq0$ almost everywhere.

The second conclusion follows then from (\ref{gmonot}) with $\varphi=u_{n+1}:$%
\[
\left[  u_{n}\right]  _{s,p}^{p}\leq\left[  u_{n+1}\right]  _{s,p}^{p}+p%
{\displaystyle\int_{\Omega}}
\dfrac{\omega_{n}(u_{n}-u_{n+1})}{(u_{n}+\frac{1}{n})^{\alpha}}\mathrm{d}%
x\leq\left[  u_{n+1}\right]  _{s,p}^{p}.
\]

\end{proof}

In what follows $\psi\in W_{0}^{s,p}(\Omega)$ is such that
\begin{equation}
\left\{
\begin{array}
[c]{lcl}%
\left(  -\Delta_{p}\right)  ^{s}\psi=\omega_{1} & \mathrm{in} & \Omega\\
\psi=0 & \mathrm{on} & \mathbb{R}^{N}\setminus\Omega.
\end{array}
\right.  \label{Psi}%
\end{equation}
Since $0\leq\omega_{1}=\min\left\{  \omega,1\right\}  \in L^{\infty}%
(\Omega)\diagdown\left\{  0\right\}  $ we can check that $\psi\in C^{\beta
_{s}}(\overline{\Omega})\ $for some $\beta_{s}\in(0,s]$ and that $\psi(x)>0$
$\forall\,x\in\Omega.$ (See arguments in the proof of Proposition \ref{un},
based on \cite{BRASCO FRANZINA,CKP,IMS}).

\begin{proposition}
\label{lowbound}Let $u_{n}\in W_{0}^{s,p}(\Omega)$ be the weak solution of
(\ref{weakun}), with $\alpha>0,$ and $\omega\in L^{1}(\Omega)\diagdown\left\{
0\right\}  ,\;\omega\geq0.$ We have%
\[
0<m_{\alpha}\psi\leq u_{1}\leq u_{n},\quad\forall\,n\in\mathbb{N},
\]
where
\[
m_{\alpha}:=(\left\Vert u_{1}\right\Vert _{\infty}+1)^{-\frac{\alpha}{p-1}}.
\]

\end{proposition}

\begin{proof}
Let $\varphi$ be any nonnegative function in $W_{0}^{s,p}(\Omega).$ Then,
\begin{align*}
\left\langle (-\Delta_{p})^{s}\,u_{1},\varphi\right\rangle  &  =%
{\displaystyle\int_{\Omega}}
\dfrac{\omega_{1}\varphi}{(u_{1}+1)^{\alpha}}\mathrm{d}x\\
&  \geq\dfrac{1}{(\left\Vert u_{1}\right\Vert _{\infty}+1)^{\alpha}}%
{\displaystyle\int_{\Omega}}
\omega_{1}\varphi\mathrm{d}x=\left\langle (-\Delta_{p})^{s}\,m_{\alpha}%
\psi,\varphi\right\rangle .\,
\end{align*}
\ 

It follows from the comparison principle for the fractional $p$-Laplacian that
$m_{\alpha}\psi\leq u_{1}.$ This concludes the proof since $u_{1}\leq u_{n}$
for all $n\in\mathbb{N}.$
\end{proof}

The following corollary is immediate since $\psi$ is strictly positive in
$\Omega$ and continuous in $\overline{\Omega}.$

\begin{corollary}
\label{cc}Let $\Omega^{\prime}$ be an arbitrary subdomain compactly contained
in $\Omega.$ There exists a positive constant $C_{\Omega^{\prime}},$ that does
not depend on $n,$ such that
\[
C_{\Omega^{\prime}}\leq u_{n}(x),\quad\forall\,x\in\Omega^{\prime}.
\]

\end{corollary}

Taking into account the monotonicity of the sequence $\left\{  u_{n}\right\}
_{n\in\mathbb{N}},$ let us define, for each $\alpha>0,$ the function
$u_{\alpha}:\overline{\Omega}\rightarrow\lbrack0,\infty]$ by
\begin{equation}
u_{\alpha}(x):=\lim_{n\rightarrow\infty}u_{n}(x)=\sup_{n\in\mathbb{N}}%
u_{n}(x). \label{ualfa}%
\end{equation}

We anticipate that $u_{\alpha}(x)<\infty$ for almost every $x\in\Omega$ (see
Remark \ref{ufinite}).

\begin{proposition}
\label{gaglibound}Let $\alpha>0$ and $\omega\in L^{1}(\Omega)\diagdown\left\{
0\right\}  ,\;\omega\geq0.$ If the sequence $\left\{  u_{n}\right\}
_{n\in\mathbb{N}}$ is bounded in $W_{0}^{s,p}(\Omega),$ then it converges in
$W_{0}^{s,p}(\Omega)$ to $u_{\alpha}$ and this function is the weak solution
of (\ref{palfa}).
\end{proposition}

\begin{proof}
We note that the condition \textrm{(i)} of Definition \ref{weaksol} is
fulfilled, according to Corollary \ref{cc}. Thus, we need to check the
condition \textrm{(ii).}

The boundedness of $\left\{  u_{n}\right\}  _{n\in\mathbb{N}}$ implies that
there exists a subsequence $\left\{  u_{n_{k}}\right\}  _{k\in\mathbb{N}}$
converging to a function $u,$ weakly in $W_{0}^{s,p}(\Omega)$ and pointwise
almost everywhere. This implies that $u=u_{\alpha}$ almost everywhere, so that
$u_{\alpha}\in W_{0}^{s,p}(\Omega).$

Thus, by applying (\ref{gmonot}) with $\varphi=u_{\alpha}$ we obtain
\[
\left[  u_{n}\right]  _{s,p}^{p}\leq\left[  u_{\alpha}\right]  _{s,p}^{p}+p%
{\displaystyle\int_{\Omega}}
\dfrac{\omega_{n}(u_{n}-u_{\alpha})}{(u_{n}+\frac{1}{n})^{\alpha}}%
\mathrm{d}x\leq\left[  u_{\alpha}\right]  _{s,p}^{p}.
\]

Combining this fact with the monotonicity of $\left\{  \left[  u_{n}\right]
_{s,p}^{p}\right\}  _{k\in\mathbb{N}}$ we get%
\[
\lim_{n\rightarrow\infty}\left[  u_{n}\right]  _{s,p}^{p}=\lim_{k\rightarrow
\infty}\left[  u_{n_{k}}\right]  _{s,p}^{p}\leq\left[  u_{\alpha}\right]
_{s,p}^{p}\leq\lim_{k\rightarrow\infty}\left[  u_{n_{k}}\right]  _{s,p}%
^{p},\quad\forall\,n\in\mathbb{N},
\]
where the latter inequality stems from the weak convergence $u_{n_{k}%
}\rightharpoonup u_{\alpha}.$

We have concluded that
\[
\left[  u_{\alpha}\right]  _{s,p}=\lim_{k\rightarrow\infty}\left[  u_{n_{k}%
}\right]  _{s,p}=\lim_{n\rightarrow\infty}\left[  u_{n}\right]  _{s,p}%
\]
and hence we obtain the strong convergence $u_{n}\rightarrow u_{\alpha}$.

This convergence and the Corollary \ref{cc} allow us to pass to the limit,
when $n\rightarrow\infty,$ in
\[
\left\langle (-\Delta_{p})^{s}\,u_{n},\varphi\right\rangle =%
{\displaystyle\int_{\Omega}}
\dfrac{\omega_{n}\varphi}{(u_{n}+\frac{1}{n})^{\alpha}}\mathrm{d}%
x,\quad\forall\,\varphi\in C_{c}^{\infty}(\Omega)
\]
in order to obtain%
\[
\left\langle (-\Delta_{p})^{s}\,u_{\alpha},\varphi\right\rangle =%
{\displaystyle\int_{\Omega}}
\dfrac{\omega\varphi}{(u_{\alpha})^{\alpha}}\mathrm{d}x,\quad\forall
\,\varphi\in C_{c}^{\infty}(\Omega).
\]
This concludes the proof that $u_{\alpha}$ is a weak solution of (\ref{palfa}).
\end{proof}

The next result is a reciprocal of Proposition \ref{gaglibound}.

\begin{proposition}
Let $\alpha>0$ and $\omega\in L^{1}(\Omega)\diagdown\left\{  0\right\}
,\;\omega\geq0.$ Suppose that $u\in W_{0}^{s,p}(\Omega)$ is a weak solution of
(\ref{palfa}). Then, $\left\{  u_{n}\right\}  _{n\in\mathbb{N}}$ converges in
$W_{0}^{s,p}(\Omega)$ to $u$ and $u=u_{\alpha}$.
\end{proposition}

\begin{proof}
Let $\varphi=(u_{n}-u)_{+}.$ On the one hand, according to Lemma \ref{lemma+},
we have
\[
\left\langle (-\Delta_{p})^{s}\,u_{n}-(-\Delta_{p})^{s}\,u,\varphi
\right\rangle \geq0.
\]
On the other hand,%
\begin{align*}
\left\langle (-\Delta_{p})^{s}\,u_{n}-(-\Delta_{p})^{s}\,u,\varphi
\right\rangle  &  =%
{\displaystyle\int_{\Omega}}
\dfrac{\omega_{n}\varphi}{(u_{n}+\frac{1}{n})^{\alpha}}-\dfrac{\omega\varphi
}{u^{\alpha}}\mathrm{d}x\\
&  \leq%
{\displaystyle\int_{\Omega}}
\dfrac{\omega_{n}\varphi}{(u_{n})^{\alpha}}-\dfrac{\omega\varphi}{u^{\alpha}%
}\mathrm{d}x\\
&  =%
{\displaystyle\int_{u_{n}\geq u}}
\dfrac{\omega_{n}\varphi}{(u_{n})^{\alpha}}-\dfrac{\omega\varphi}{u^{\alpha}%
}\mathrm{d}x\\
&  \leq%
{\displaystyle\int_{u_{n}\geq u}}
\left(  \dfrac{1}{(u_{n})^{\alpha}}-\dfrac{1}{u^{\alpha}}\right)
\omega\varphi\mathrm{d}x\leq0.
\end{align*}
Thus, by repeating the arguments in the proof of Proposition \ref{unmonot} we
can conclude that $u_{n}\leq u$ almost everywhere.

Hence, by using (\ref{gmonot}), we obtain the boundedness of the sequence
$\left\{  u_{n}\right\}  _{n\in\mathbb{N}}$ in $W_{0}^{s,p}(\Omega):$
\[
\left[  u_{n}\right]  _{s,p}^{p}\leq\left[  u\right]  _{s,p}^{p}+p%
{\displaystyle\int_{\Omega}}
\dfrac{\omega_{n}(u_{n}-u)}{(u_{n}+\frac{1}{n})^{\alpha}}\mathrm{d}%
x\leq\left[  u\right]  _{s,p}^{p}.
\]

Consequently, according to Proposition \ref{gaglibound}, $\left\{
u_{n}\right\}  _{n\in\mathbb{N}}$ converges in $W_{0}^{s,p}(\Omega)$ to
$u_{\alpha}$ and this function is the only solution of (\ref{palfa}).
Therefore, $u=u_{\alpha}.$
\end{proof}

\subsection{Existence for the singular problem}

In the sequel we will use the following notation
\begin{equation}
r_{\alpha}:=\left\{
\begin{array}
[c]{ccl}%
1 & \mathrm{if} & \alpha=1\\
\left(  \dfrac{p_{s}^{\star}}{1-\alpha}\right)  ^{\prime} & \mathrm{if} &
0<\alpha<1\quad\mathrm{and}\quad sp<N\\
\alpha^{-1} & \mathrm{if} & 0<\alpha<1\quad\mathrm{and}\quad sp\geq N.
\end{array}
\right.  \label{qalpha}%
\end{equation}

\begin{theorem}
\label{existence}Let $0<\alpha\leq1$ and $\omega\in L^{r}(\Omega),$ with
$r\geq r_{\alpha}.$ The sequence $\left\{  u_{n}\right\}  _{n\in\mathbb{N}}$
is bounded in $W_{0}^{s,p}(\Omega).$ Consequently, it converges in
$W_{0}^{s,p}(\Omega)$ to $u_{\alpha}$ and this function is the weak solution
of (\ref{palfa}).
\end{theorem}

\begin{proof}
We will assume in this proof, without loss of generality, that $r=r_{\alpha}$
(note that $L^{r}(\Omega)\hookrightarrow L^{r_{\alpha}}(\Omega)$ whenever
$r\geq r_{\alpha}$).

According to Proposition \ref{gaglibound}, we need only to show that the
sequence $\left\{  u_{n}\right\}  _{n\in\mathbb{N}}$ is bounded in
$W_{0}^{s,p}(\Omega).$

We have%
\begin{equation}
\left[  u_{n}\right]  _{s,p}^{p}=%
{\displaystyle\int_{\Omega}}
\dfrac{\omega_{n}u_{n}}{(u_{n}+\frac{1}{n})^{\alpha}}\mathrm{d}x\leq%
{\displaystyle\int_{\Omega}}
\dfrac{\omega u_{n}}{(u_{n}+\frac{1}{n})^{\alpha}}\mathrm{d}x\leq%
{\displaystyle\int_{\Omega}}
(u_{n})^{1-\alpha}\omega\mathrm{d}x, \label{aux9}%
\end{equation}
where the equality follows from (\ref{weakun}). Thus, $\left[  u_{n}\right]
_{s,p}^{p}\leq\left\Vert \omega\right\Vert _{1}=\left\Vert \omega\right\Vert
_{r_{\alpha}},$ when $\alpha=1.$

In the case $0<\alpha<1,$ by applying H\"{o}lder inequality to (\ref{aux9}),
we obtain%
\begin{equation}
\left[  u_{n}\right]  _{s,p}^{p}\leq\left(
{\displaystyle\int_{\Omega}}
(u_{n})^{(1-\alpha)(r_{\alpha})^{\prime}}\mathrm{d}x\right)  ^{\frac
{1}{(r_{\alpha})^{\prime}}}\left(
{\displaystyle\int_{\Omega}}
\left\vert \omega\right\vert ^{r_{\alpha}}\mathrm{d}x\right)  ^{\frac
{1}{r_{\alpha}}}=\left\Vert u_{n}\right\Vert _{(1-\alpha)(r_{\alpha})^{\prime
}}^{1-\alpha}\left\Vert \omega\right\Vert _{r_{\alpha}}. \label{aux10}%
\end{equation}
Hence, when $sp<N$ we have $(1-\alpha)(r_{\alpha})^{\prime}=p_{s}^{\star}$, so
that%
\[
\left[  u_{n}\right]  _{s,p}^{p}\leq\left\Vert \omega\right\Vert _{r_{\alpha}%
}\left\Vert u_{n}\right\Vert _{p_{s}^{\star}}^{1-\alpha}\leq\left(
(S_{p_{s}^{\star}})^{\frac{1}{p}}\left[  u_{n}\right]  _{s,p}\right)
^{1-\alpha}\left\Vert \omega\right\Vert _{r_{\alpha}}.
\]
It follows that $\left\{  u_{n}\right\}  _{n\in\mathbb{N}}$ is bounded in
$W_{0}^{s,p}(\Omega)$ and
\[
\left[  u_{n}\right]  _{s,p}^{p-(1-\alpha)}\leq\left\Vert \omega\right\Vert
_{r_{\alpha}}(S_{p_{s}^{\star}})^{\frac{1-\alpha}{p}}.
\]

At last, for $sp\geq N$ we have $(1-\alpha)(r_{\alpha})^{\prime}=1,$ so that,
by (\ref{aux10}),%
\[
\left[  u_{n}\right]  _{s,p}^{p}\leq\left\Vert \omega\right\Vert _{r_{\alpha}%
}\left\Vert u_{n}\right\Vert _{1}^{1-\alpha}\leq\left\Vert \omega\right\Vert
_{r_{\alpha}}\left(  (S_{1})^{\frac{1}{p}}\left[  u_{n}\right]  _{s,p}\right)
^{1-\alpha}.
\]
Therefore, $\left\{  u_{n}\right\}  _{n\in\mathbb{N}}$ is bounded in
$W_{0}^{s,p}(\Omega)\ $and
\[
\left[  u_{n}\right]  _{s,p}^{p-(1-\alpha)}\leq\left\Vert \omega\right\Vert
_{r_{\alpha}}(S_{1})^{\frac{1-\alpha}{p}}.
\]

\end{proof}

\begin{remark}
\label{ufinite}Theorem \ref{existence} guarantees that if $0<\alpha\leq1$ and
$\omega\in L^{r}(\Omega),$ with $r\geq r_{\alpha},$ then $u_{\alpha}%
(x)<\infty$ for almost every $x\in\Omega.$ The same holds true if $\alpha>1$
and $\omega\in L^{1}(\Omega).$ Indeed, in \cite[Lemma 3.4]{CMS} the authors
proved that, under these hypotheses, the sequence $\left\{  u_{n}%
^{(\alpha-1+p)/p}\right\}  _{n\in\mathbb{N}}$ is bounded in $W_{0}%
^{s,p}(\Omega).$ This fact and the monotonicity of $\left\{  u_{n}\right\}
_{n\in\mathbb{N}}$ imply that $u_{\alpha}^{(\alpha-1+p)/p}\in L^{1}(\Omega),$
so that $u_{\alpha}(x)<\infty$ for almost every $x\in\Omega.$
\end{remark}

\begin{remark}
\label{alfa>1}When $\alpha>1,$ we have
\[
\left[  u_{n}\right]  _{s,p}^{p}=%
{\displaystyle\int_{\Omega}}
\dfrac{\omega_{n}u_{n}}{(u_{n}+\frac{1}{n})^{\alpha}}\mathrm{d}x\leq%
{\displaystyle\int_{\Omega}}
\dfrac{\omega u_{n}}{(u_{n})^{\alpha}}\mathrm{d}x\leq%
{\displaystyle\int_{\Omega}}
\dfrac{\omega}{(u_{1})^{\alpha-1}}\mathrm{d}x.
\]
Thus, if $\omega$ belongs to $L^{1}(\Omega)$ and vanishes in $\Omega
\setminus\Omega^{\prime},$ for some proper subdomain $\Omega^{\prime}$ of
$\Omega,$ then
\[
\left[  u_{n}\right]  _{s,p}^{p}\leq%
{\displaystyle\int_{\Omega^{\prime}}}
\dfrac{\omega}{(u_{1})^{\alpha-1}}\mathrm{d}x\leq\frac{\left\Vert
\omega\right\Vert _{1}}{\min_{\overline{\Omega^{\prime}}}(u_{1})^{\alpha-1}%
}<\infty,
\]
which shows that $u_{\alpha}$ is the only weak solution of (\ref{palfa}).
\end{remark}

\section{Sobolev inequality associated with $0<\alpha<1$\label{Sec3}}

In this section we consider $0<\alpha<1$ and $\omega\in L^{r}(\Omega),$ with
$r\geq r_{\alpha},$ where $r_{\alpha}$ is defined by (\ref{qalpha}). Thus,
according to Theorem \ref{existence}, the existence of the unique weak
solution $u_{\alpha}$ of the singular problem (\ref{palfa}) is guaranteed.

In order to derive the Sobolev inequality (\ref{Sobolev1}) we will first show
that $u_{\alpha}$ minimizes the energy functional $E_{\alpha}:W_{0}%
^{s,p}(\Omega)\longrightarrow\mathbb{R},$ associated with the singular problem
(\ref{palfa}), defined by%
\[
E_{\alpha}(v):=\frac{1}{p}\left[  v\right]  _{s,p}^{p}-\frac{1}{1-\alpha}%
{\displaystyle\int_{\Omega}}
(v_{+})^{1-\alpha}\omega\mathrm{d}x.
\]

Since $E_{\alpha}$ is not differentiable we will obtain its minimizer as the
limit of the sequence $\left\{  u_{n}\right\}  _{n\in\mathbb{N}}$ by taking
advantage that $u_{n}$ minimizes the energy functional $E_{n}:W_{0}%
^{s,p}(\Omega)\longrightarrow\mathbb{R}$ associated with (\ref{Pn}), which is
defined by
\[
E_{n}(v):=\frac{1}{p}\left[  v\right]  _{s,p}^{p}-%
{\displaystyle\int_{\Omega}}
G_{n}(v)\omega_{n}\mathrm{d}x
\]
where $G_{n}:\mathbb{R}\rightarrow\mathbb{R}$ is the increasing $C^{1}$
function
\[
G_{n}(t):=\frac{1}{1-\alpha}(t_{+}+\frac{1}{n})^{1-\alpha}-\left(  \frac{1}%
{n}\right)  ^{-\alpha}t_{-}%
\]
(as usual, $t_{\pm}=\max\left\{  \pm t,0\right\}  $).

One can easily see that $E_{n}$ is of class $C^{1}$ and
\[
\left\langle E_{n}^{\prime}(v),\varphi\right\rangle =\left\langle (-\Delta
_{p})^{s}\,v,\varphi\right\rangle -%
{\displaystyle\int_{\Omega}}
G_{n}^{\prime}(v)\omega_{n}\varphi\mathrm{d}x,\quad\forall\,\varphi\in
W_{0}^{s,p}(\Omega).
\]
Thus, nonnegative critical points of $E_{n}$ are weak solutions of (\ref{Pn}).
Moreover, by making use of standard arguments one can also check that $E_{n}$
is coercive and bounded from below. All of these features of $E_{n}$ allow one
to verify that $E_{n}$ attains its minimum value at a function $v_{n}\in
W_{0}^{s,p}(\Omega).$ Since $E_{n}(v_{+})\leq E_{n}(v)$ for all $v\in
W_{0}^{s,p}(\Omega)$ one has $v_{n}\geq0.$ Of course, the minimizer $v_{n}$ is
also a critical point of $E_{n},$ that is,%
\[
\left\langle (-\Delta_{p})^{s}\,v_{n},\varphi\right\rangle =%
{\displaystyle\int_{\Omega}}
\frac{\omega_{n}\varphi}{(v_{n}+\frac{1}{n})^{\alpha}}\mathrm{d}x,\quad
\forall\,\varphi\in W_{0}^{s,p}(\Omega).
\]
Therefore, $v_{n}=u_{n}$ since $u_{n}$ is the only nonnegative function
satisfying (\ref{weakun}).

\begin{proposition}
The function $u_{\alpha}$ minimizes the energy functional $E_{\alpha}.$
\end{proposition}

\begin{proof}
Recall that $u_{n}\rightarrow u_{\alpha}$ strongly in $W_{0}^{s,p}(\Omega)$
and that $u_{n}\leq u_{\alpha}.$ Thus, $\left[  u_{n}\right]  _{s,p}%
^{p}\rightarrow\left[  u_{\alpha}\right]  _{s,p}^{p},$
\[
0\leq G_{n}(u_{n})\omega_{n}\leq\frac{(u_{\alpha}+1)^{1-\alpha}\omega
}{1-\alpha}\in L^{1}(\Omega),
\]
and%
\[
\lim_{n\rightarrow\infty}G_{n}(u_{n})\omega_{n}=\frac{(u_{\alpha})^{1-\alpha
}\omega}{1-\alpha}\quad\mathrm{a.e.\,in}\,\Omega.
\]
These facts show that $E_{n}(u_{n})\rightarrow E_{\alpha}(u_{\alpha}).$

For each $v\in W_{0}^{s,p}(\Omega)$ we have%
\[
0\leq G_{n}(v_{+})\omega_{n}\leq\frac{(v_{+}+1)^{1-\alpha}\omega}{1-\alpha}\in
L^{1}(\Omega),
\]
and%
\[
\lim_{n\rightarrow\infty}G_{n}(v_{+})\omega_{n}=\frac{(v_{+})^{1-\alpha}%
\omega}{1-\alpha},\quad\mathrm{a.e.\,in}\,\Omega,
\]
so that $E_{n}(v)\rightarrow E_{\alpha}(v).$

Therefore, observing that $E_{n}(u_{n})\leq E_{n}(v)$ we obtain%
\[
E_{\alpha}(u_{\alpha})\leq E_{\alpha}(v),\quad\forall\,v\in W_{0}^{s,p}%
(\Omega).
\]

\end{proof}

In order to simply the notation in the sequence, let us define%
\[
\mathcal{M}_{\alpha}:=\left\{  v\in W_{0}^{s,p}(\Omega):%
{\displaystyle\int_{\Omega}}
\left\vert v\right\vert ^{1-\alpha}\omega\mathrm{d}x=1\right\}
\]
and%
\begin{equation}
U_{\alpha}:=\theta_{\alpha}u_{\alpha},\quad\mathrm{where}\quad\theta_{\alpha
}:=\left(
{\displaystyle\int_{\Omega}}
\left\vert u_{\alpha}\right\vert ^{1-\alpha}\omega\mathrm{d}x\right)
^{-\frac{1}{1-\alpha}}. \label{Ualfa}%
\end{equation}

Of course, $U_{\alpha}\in\mathcal{M}_{\alpha}.$

\begin{theorem}
\label{teo1}One has
\begin{equation}
\left[  U_{\alpha}\right]  _{s,p}^{p}=\left[  u_{\alpha}\right]
_{s,p}^{p(\frac{1-\alpha-p}{1-\alpha})}=\min_{v\in\mathcal{M}_{\alpha}}\left[
v\right]  _{s,p}^{p}. \label{Ray}%
\end{equation}

\end{theorem}

\begin{proof}
Since $u_{\alpha}$ is a weak solution of (\ref{palfa}) we have%
\begin{equation}
\left[  u_{\alpha}\right]  _{s,p}^{p}=%
{\displaystyle\int_{\Omega}}
\left\vert u_{\alpha}\right\vert ^{1-\alpha}\omega\mathrm{d}x, \label{aux13}%
\end{equation}
so that
\begin{align*}
\left[  U_{\alpha}\right]  _{s,p}^{p}  &  =(\theta_{\alpha})^{p}\left[
u_{\alpha}\right]  _{s,p}^{p}\\
&  =\left(
{\displaystyle\int_{\Omega}}
\left\vert u_{\alpha}\right\vert ^{1-\alpha}\omega\mathrm{d}x\right)
^{-\frac{p}{1-\alpha}}\left[  u_{\alpha}\right]  _{s,p}^{p}=\left(  \left[
u_{\alpha}\right]  _{s,p}^{p}\right)  ^{-\frac{p}{1-\alpha}}\left[  u_{\alpha
}\right]  _{s,p}^{p}=\left[  u_{\alpha}\right]  _{s,p}^{p(\frac{1-\alpha
-p}{1-\alpha})},
\end{align*}
what is the first equality in (\ref{Ray}).

In order to prove the second equality in (\ref{Ray}) let us fix $v\in
\mathcal{M}_{\alpha}.$ It follows from (\ref{aux13}) that%
\[
E_{\alpha}(u_{\alpha})=\frac{1}{p}\left[  u_{\alpha}\right]  _{s,p}^{p}%
-\frac{1}{1-\alpha}%
{\displaystyle\int_{\Omega}}
(u_{\alpha})_{+}^{1-\alpha}\omega\mathrm{d}x=\left(  \frac{1}{p}-\frac
{1}{1-\alpha}\right)  \left[  u_{\alpha}\right]  _{s,p}^{p}.
\]

Now, for any $t>0$ we have
\begin{align*}
\left(  \frac{1}{p}-\frac{1}{1-\alpha}\right)  \left[  u_{\alpha}\right]
_{s,p}^{p}  &  =E_{\alpha}(u_{\alpha})\\
&  \leq E_{\alpha}(t\left\vert v\right\vert )\\
&  =\frac{t^{p}}{p}\left[  \left\vert v\right\vert \right]  _{s,p}^{p}%
-\frac{t^{1-\alpha}}{1-\alpha}%
{\displaystyle\int_{\Omega}}
\left\vert v\right\vert ^{1-\alpha}\omega\mathrm{d}x\\
&  =\frac{t^{p}}{p}\left[  \left\vert v\right\vert \right]  _{s,p}^{p}%
-\frac{t^{1-\alpha}}{1-\alpha}\leq\frac{t^{p}}{p}\left[  v\right]  _{s,p}%
^{p}-\frac{t^{1-\alpha}}{1-\alpha},
\end{align*}
that is
\[
t^{1-\alpha}\left(  \frac{1}{1-\alpha}-\frac{t^{p-(1-\alpha)}}{p}\left[
v\right]  _{s,p}^{p}\right)  \leq\left(  \frac{1}{1-\alpha}-\frac{1}%
{p}\right)  \left[  u_{\alpha}\right]  _{s,p}^{p}.
\]

By choosing
\[
t=\left(  \left[  v\right]  _{s,p}^{p}\right)  ^{-\frac{1}{p-(1-\alpha)}}%
\]
we obtain%
\[
\left(  \left[  v\right]  _{s,p}^{p}\right)  ^{-\frac{1-\alpha}{p-(1-\alpha)}%
}\left(  \frac{1}{1-\alpha}-\frac{1}{p}\right)  \leq\left(  \frac{1}{1-\alpha
}-\frac{1}{p}\right)  \left[  u_{\alpha}\right]  _{s,p}^{p},
\]
so that%
\[
\left[  u_{\alpha}\right]  _{s,p}^{p(\frac{1-\alpha-p}{1-\alpha})}\leq\left[
v\right]  _{s,p}^{p}.
\]

This fact implies that
\[
\left[  u_{\alpha}\right]  _{s,p}^{p(\frac{1-\alpha-p}{1-\alpha})}\leq
\inf_{v\in\mathcal{M}_{\alpha}}\left[  v\right]  _{s,p}^{p}%
\]
and then the first equality in (\ref{Ray}) shows that this infimum is reached
at $U_{\alpha}.$
\end{proof}

From now on we denote the minimum in (\ref{Ray}) by $\lambda_{\alpha},$ that
is,%
\begin{equation}
\lambda_{\alpha}:=\min_{v\in\mathcal{M}_{\alpha}}\left[  v\right]  _{s,p}%
^{p}=\left[  U_{\alpha}\right]  _{s,p}^{p}=\left[  u_{\alpha}\right]
_{s,p}^{p(\frac{1-\alpha-p}{1-\alpha})}. \label{lambalf}%
\end{equation}

\begin{corollary}
The inequality
\begin{equation}
C\left(
{\displaystyle\int_{\Omega}}
\left\vert v\right\vert ^{1-\alpha}\omega\mathrm{d}x\right)  ^{\frac
{p}{1-\alpha}}\leq\left[  v\right]  _{s,p}^{p},\quad\forall\,v\in W_{0}%
^{s,p}(\Omega) \label{Sobolev1a}%
\end{equation}
holds if, and only if, $C\leq\lambda_{\alpha}.$
\end{corollary}

\begin{proof}
Since
\[
\left(
{\displaystyle\int_{\Omega}}
\left\vert v\right\vert ^{1-\alpha}\omega\mathrm{d}x\right)  ^{-\frac
{1}{1-\alpha}}v\in\mathcal{M}_{\alpha},\quad\forall\,v\in W_{0}^{s,p}%
(\Omega)\setminus\left\{  0\right\}
\]
it follows from Theorem \ref{teo1} that (\ref{Sobolev1a}) holds for any
$C\leq\lambda_{a}.$ We can see from (\ref{lambalf}) that if $C>\lambda
_{\alpha}$ then (\ref{Sobolev1a}) fails at some $v\in\mathcal{M}_{\alpha}.$
\end{proof}

\begin{proposition}
\label{uniqmin}The only minimizers of the functional $v\mapsto\lbrack
v]_{s,p}^{p}$ on $\mathcal{M}_{\alpha}$ are $U_{\alpha}$ and $-U_{\alpha}.$
Therefore, if
\begin{equation}
\lambda_{\alpha}\left(
{\displaystyle\int_{\Omega}}
\left\vert v\right\vert ^{1-\alpha}\omega\mathrm{d}x\right)  ^{\frac
{p}{1-\alpha}}=\left[  v\right]  _{s,p}^{p} \label{Sob1b}%
\end{equation}
for some $v\in W_{0}^{s,p}(\Omega)\setminus\left\{  0\right\}  ,$ then
$v=kU_{\alpha}$ for some constant $k.$
\end{proposition}

\begin{proof}
Let $\Phi\in\mathcal{M}_{\alpha}$ be such that $\lambda_{\alpha}=\left[
\Phi\right]  _{s,p}^{p}.$ We observe from Remark \ref{Sign} that $\Phi$ does
not change sign in $\Omega.$ Indeed, otherwise we would arrive at the
following absurd, since $\left\vert \Phi\right\vert \in\mathcal{M}_{\alpha}:$%
\[
\left[  \left\vert \Phi\right\vert \right]  _{s,p}^{p}<\left[  \Phi\right]
_{s,p}^{p}=\lambda_{\alpha}\leq\left[  \left\vert \Phi\right\vert \right]
_{s,p}^{p}.
\]
Thus, without loss of generality, we assume that $\Phi\geq0$ in $\Omega$
(otherwise, we proceed with $-\Phi$ instead of $\Phi$).

Since $\Phi,$ $U_{\alpha},$ $\omega\geq0$ and $0<1-\alpha<1$ we have%
\begin{align*}
\left(  \int_{\Omega}\left(  \frac{\Phi}{2}+\frac{U_{\alpha}}{2}\right)
^{1-\alpha}\omega\mathrm{d}x\right)  ^{\frac{1}{1-\alpha}}  &  =\left(
\int_{\Omega}\left(  \frac{\Phi}{2}\omega^{\frac{1}{1-\alpha}}+\frac
{U_{\alpha}}{2}\omega^{\frac{1}{1-\alpha}}\right)  ^{1-\alpha}\mathrm{d}%
x\right)  ^{\frac{1}{1-\alpha}}\\
&  \geq\left(  \int_{\Omega}\left(  \frac{\Phi}{2}\right)  ^{1-\alpha}%
\omega\mathrm{d}x\right)  ^{\frac{1}{1-\alpha}}+\left(  \int_{\Omega}\left(
\frac{U_{\alpha}}{2}\right)  ^{1-\alpha}\omega\mathrm{d}x\right)  ^{\frac
{1}{1-\alpha}}\\
&  =\frac{1}{2}\left(  \int_{\Omega}\Phi^{1-\alpha}\omega\mathrm{d}x\right)
^{\frac{1}{1-\alpha}}+\frac{1}{2}\left(  \int_{\Omega}\frac{U_{\alpha
}^{1-\alpha}}{2}\omega\mathrm{d}x\right)  ^{\frac{1}{1-\alpha}}\\
&  =\frac{1}{2}+\frac{1}{2}=1,
\end{align*}
showing that
\[
h:=\left(  \int_{\Omega}\left(  \frac{\Phi}{2}+\frac{U_{\alpha}}{2}\right)
^{1-\alpha}\omega\mathrm{d}x\right)  ^{\frac{1}{1-\alpha}}\geq1.
\]

Observing that $h^{-1}(\frac{1}{2}\Phi+\frac{1}{2}U_{\alpha})\in
\mathcal{M}_{\alpha}$ and
\begin{align*}
\lambda_{\alpha}  &  \leq\left[  h^{-1}(\frac{\Phi}{2}+\frac{U_{\alpha}}%
{2})\right]  _{s,p}^{p}\\
&  \leq\frac{1}{h^{p}}\left(  \left[  \frac{\Phi}{2}\right]  _{s,p}+\left[
\frac{U_{\alpha}}{2}\right]  _{s,p}\right)  ^{p}=\frac{1}{h^{p}}\left(
\frac{\lambda_{\alpha}^{\frac{1}{p}}}{2}+\frac{\lambda_{\alpha}^{\frac{1}{p}}%
}{2}\right)  ^{p}=\frac{\lambda_{\alpha}}{h^{p}}\leq\lambda_{\alpha}%
\end{align*}
we can conclude that: $h=1,$ $\frac{1}{2}\Phi+\frac{1}{2}U_{\alpha}%
\in\mathcal{M}_{\alpha}$ and
\begin{equation}
\lambda_{\alpha}^{\frac{1}{p}}=\left[  \frac{\Phi}{2}+\frac{U_{\alpha}}%
{2}\right]  _{s,p}=\left(  \frac{\left[  \Phi\right]  _{s,p}}{2}+\frac{\left[
U_{\alpha}\right]  _{s,p}}{2}\right)  . \label{Sob1c}%
\end{equation}

We recall that the functional $v\mapsto\left[  v\right]  _{s,p}$ is strictly
convex over $W_{0}^{s,p}(\Omega).$ Thus, the second equality in (\ref{Sob1c})
implies that $\Phi=U_{\alpha}.$

We have shown that $\lambda_{\alpha}=\left[  \Phi\right]  _{s,p}^{p}$ for some
$\Phi\in\mathcal{M}_{\alpha}$ if, and only if, either $\Phi=U_{\alpha}$ or
$\Phi=-U_{\alpha}.$ Thus, if (\ref{Sob1b}) holds true for some $v\in
W_{0}^{s,p}(\Omega)\setminus\left\{  0\right\}  ,$ then either $v=\theta
^{-1}U_{\alpha}$ or $v=-\theta^{-1}U_{\alpha},$ where $\theta=\left(
{\displaystyle\int_{\Omega}}
\left\vert v\right\vert ^{1-\alpha}\omega\mathrm{d}x\right)  ^{-\frac
{1}{1-\alpha}}$ (since $\Phi=\theta v\in\mathcal{M}_{\alpha}$ and
$\lambda_{\alpha}=\left[  \theta v\right]  _{s,p}^{p}$).
\end{proof}

\section{Sobolev inequality associated with $\alpha=1$\label{Sec4}}

According to (\ref{lambalf})%
\begin{equation}
\lambda_{\alpha}\left\Vert \omega\right\Vert _{1}^{\frac{p}{1-\alpha}}\left(
\frac{1}{\left\Vert \omega\right\Vert _{1}}%
{\displaystyle\int_{\Omega}}
\left\vert v\right\vert ^{1-\alpha}\omega\mathrm{d}x\right)  ^{\frac
{p}{1-\alpha}}\leq\left[  v\right]  _{s,p}^{p},\quad\forall\,v\in W_{0}%
^{s,p}(\Omega). \label{Saw1}%
\end{equation}

We would like to pass to the limit, as $\alpha\rightarrow1^{-},$ in the above
inequality. For this, we need the following two lemmas.

\begin{lemma}
\label{monot}Let $\omega\in L^{r}(\Omega),$ $r>1,$ and $v\in W_{0}%
^{s,p}(\Omega).$ The map
\begin{equation}
(0,\frac{p_{s}^{\star}}{r^{\prime}})\ni q\mapsto\left(  \frac{1}{\left\Vert
\omega\right\Vert _{1}}%
{\displaystyle\int_{\Omega}}
\left\vert v\right\vert ^{q}\omega\mathrm{d}x\right)  ^{\frac{1}{q}}
\label{mapq}%
\end{equation}
is well-defined and nondecreasing.
\end{lemma}

\begin{proof}
For simplicity, let us denote $\overline{\omega}=\frac{\omega}{\left\Vert
\omega\right\Vert _{1}},$ so that $\left\Vert \overline{\omega}\right\Vert
_{1}=1.$ For each $q\in(0,\frac{p_{s}^{\star}}{r^{\prime}})$ we have, by
H\"{o}lder's inequality,
\[
\left(
{\displaystyle\int_{\Omega}}
\left\vert v\right\vert ^{q}\overline{\omega}\mathrm{d}x\right)  ^{\frac{1}%
{q}}\leq\left\Vert \overline{\omega}\right\Vert _{r}^{1/q}\left\Vert
v\right\Vert _{qr^{\prime}}<\infty.
\]
Therefore, since the embedding $W_{0}^{s,p}(\Omega)\hookrightarrow
L^{qr^{\prime}}(\Omega)$ is continuous, the map (\ref{mapq}) is well-defined.

Now, in order to prove the monotonicity of this map, let $0<q_{1}<q_{2}%
<\frac{p_{s}^{\star}}{r^{\prime}}.$ By H\"{o}lder's inequality we have
\begin{align*}%
{\displaystyle\int_{\Omega}}
\left\vert v\right\vert ^{q_{1}}\overline{\omega}\mathrm{d}x  &  =%
{\displaystyle\int_{\Omega}}
\left\vert v\right\vert ^{q_{1}}\overline{\omega}^{\frac{q_{1}}{q_{2}}%
}\overline{\omega}^{\frac{q_{2-}q_{1}}{q_{2}}}\mathrm{d}x\\
&  \leq\left(
{\displaystyle\int_{\Omega}}
\left(  \left\vert v\right\vert ^{q_{1}}\overline{\omega}^{\frac{q_{1}}{q_{2}%
}}\right)  ^{\frac{q_{2}}{q_{1}}}\mathrm{d}x\right)  ^{\frac{q_{1}}{q_{2}}%
}\left(
{\displaystyle\int_{\Omega}}
\left(  \overline{\omega}^{\frac{q_{2-}q_{1}}{q_{2}}}\right)  ^{\frac{q_{2}%
}{q_{2}-q_{1}}}\right)  ^{\frac{q_{2}-q_{1}}{q_{2}}}\mathrm{d}x\\
&  =\left(
{\displaystyle\int_{\Omega}}
\left\vert v\right\vert ^{q_{2}}\overline{\omega}\mathrm{d}x\right)
^{\frac{q_{1}}{q_{2}}}\left\Vert \overline{\omega}\right\Vert _{1}%
^{1-\frac{q_{2}}{q_{1}}}=\left(
{\displaystyle\int_{\Omega}}
\left\vert v\right\vert ^{q_{2}}\overline{\omega}\mathrm{d}x\right)
^{\frac{q_{1}}{q_{2}}},
\end{align*}
implying that%
\[
\left(
{\displaystyle\int_{\Omega}}
\left\vert v\right\vert ^{q_{1}}\overline{\omega}\mathrm{d}x\right)
^{\frac{1}{q_{1}}}\leq\left(
{\displaystyle\int_{\Omega}}
\left\vert v\right\vert ^{q_{2}}\overline{\omega}\mathrm{d}x\right)
^{\frac{1}{q_{2}}}.
\]

\end{proof}

\begin{lemma}
\label{mu<ray}Let $\omega\in L^{r}(\Omega),$ $r>1.$ The map
\[
\lbrack\alpha_{0},1)\ni\alpha\mapsto\lambda_{\alpha}\left\Vert \omega
\right\Vert _{1}^{\frac{p}{1-\alpha}}%
\]
is nondecreasing, for some $\alpha_{0}\in(0,1).$
\end{lemma}

\begin{proof}
Since $\lim_{\alpha\rightarrow1^{-}}r_{\alpha}=1^{+},$ there exists
$\alpha_{0}\in(0,1)$ such that $r\geq r_{\alpha}$ whenever $\alpha\in
\lbrack\alpha_{0},1).$ Thus, according to Section \ref{Sec3}, for each
$\alpha\in\lbrack\alpha_{0},1)$ there exists $u_{\alpha}\in W_{0}^{s,p}%
(\Omega)$ such that
\[
\lambda_{\alpha}=\frac{\left[  u_{\alpha}\right]  _{s,p}^{p}}{\left(
{\displaystyle\int_{\Omega}}
\left\vert u_{\alpha}\right\vert ^{1-\alpha}\omega\mathrm{d}x\right)
^{\frac{p}{1-\alpha}}}\leq\frac{\left[  v\right]  _{s,p}^{p}}{\left(
{\displaystyle\int_{\Omega}}
\left\vert v\right\vert ^{1-\alpha}\omega\mathrm{d}x\right)  ^{\frac
{p}{1-\alpha}}},\quad\forall\,v\in W_{0}^{s,p}(\Omega).
\]

Now, let $\alpha_{0}\leq\alpha_{1}<\alpha_{2}<1.$ We have
\begin{align*}
\lambda_{\alpha_{1}}\left\Vert \omega\right\Vert _{1}^{\frac{p}{1-\alpha_{1}%
}}  &  \leq\frac{\left[  u_{\alpha_{2}}\right]  _{s,p}^{p}}{\left(  \frac
{1}{\left\Vert \omega\right\Vert _{1}}%
{\displaystyle\int_{\Omega}}
\left\vert u_{\alpha_{2}}\right\vert ^{1-\alpha_{1}}\omega\mathrm{d}x\right)
^{\frac{p}{1-\alpha_{1}}}}\\
&  \leq\frac{\left[  u_{\alpha_{2}}\right]  _{s,p}^{p}}{\left(  \frac
{1}{\left\Vert \omega\right\Vert _{1}}%
{\displaystyle\int_{\Omega}}
\left\vert u_{\alpha_{2}}\right\vert ^{1-\alpha_{2}}\omega\mathrm{d}x\right)
^{\frac{p}{1-\alpha_{2}}}}=\lambda_{\alpha_{2}}\left\Vert \omega\right\Vert
_{1}^{\frac{p}{1-\alpha}},
\end{align*}
where the second inequality comes from Lemma \ref{monot}.
\end{proof}

\begin{remark}
\label{monota}L'H\^{o}pital's rule and Lemma \ref{monot} show that%
\begin{align*}
0  &  \leq\exp\left(  \frac{p}{\left\Vert \omega\right\Vert _{1}}%
{\displaystyle\int_{\Omega}}
(\log\left\vert v\right\vert )\omega\mathrm{d}x\right) \\
&  =\lim_{q\rightarrow0^{+}}\left(  \frac{1}{\left\Vert \omega\right\Vert
_{1}}%
{\displaystyle\int_{\Omega}}
\left\vert v\right\vert ^{q}\omega\mathrm{d}x\right)  ^{\frac{p}{q}}\\
&  \leq\left(  \frac{1}{\left\Vert \omega\right\Vert _{1}}%
{\displaystyle\int_{\Omega}}
\left\vert v\right\vert ^{1-\alpha}\omega\mathrm{d}x\right)  ^{\frac
{p}{1-\alpha}}<\infty,\quad\forall\,v\in W_{0}^{s,p}(\Omega)\quad
\mathrm{and}\quad\alpha\in\lbrack\alpha_{0},1).
\end{align*}

\end{remark}

As a consequence of Lemma \ref{mu<ray}, we can define%
\[
\mu:=\lim_{\alpha\rightarrow1^{-}}\lambda_{\alpha}\left\Vert \omega\right\Vert
_{1}^{\frac{p}{1-\alpha}}=\sup_{t\in\lbrack\alpha_{0},1)}\lambda_{t}\left\Vert
\omega\right\Vert _{1}^{\frac{p}{1-t}}%
\]
and also conclude that%
\[
\mu\geq\lambda_{\alpha_{0}}\left\Vert \omega\right\Vert _{1}^{\frac
{p}{1-\alpha_{0}}}>0.
\]

However, we cannot guarantee, at least in principle, that $\mu<\infty.$
According to (\ref{Saw1}), one way of achieving this is to show the existence
of a function $\varphi\in W_{0}^{s,p}(\Omega)$ satisfying%
\begin{equation}
\lim_{q\rightarrow0^{+}}\left(  \frac{1}{\left\Vert \omega\right\Vert _{1}}%
{\displaystyle\int_{\Omega}}
\left\vert \varphi\right\vert ^{q}\omega\mathrm{d}x\right)  ^{\frac{1}{q}}>0,
\label{lim+}%
\end{equation}
or, equivalently,
\begin{equation}
-\infty<%
{\displaystyle\int_{\Omega}}
(\log\left\vert \varphi\right\vert )\omega\mathrm{d}x. \label{log+}%
\end{equation}

Apparently, the task of finding such a function $\varphi\in W_{0}^{s,p}%
(\Omega)$ is not simple when a general nonnegative function $\omega\in
L^{r}(\Omega)$ is considered. Note, for example, that if $v\in W_{0}%
^{s,p}(\Omega)\cap L^{\infty}(\Omega)$ and vanishes over a part $\Omega
^{\prime}$ of the support of $\omega$ in such a way that $0<%
{\displaystyle\int_{\Omega\setminus\Omega^{\prime}}}
\omega\mathrm{d}x<\left\Vert \omega\right\Vert _{1},$ then
\begin{align*}
\left(  \frac{1}{\left\Vert \omega\right\Vert _{1}}%
{\displaystyle\int_{\Omega}}
\left\vert v\right\vert ^{q}\omega\mathrm{d}x\right)  ^{\frac{1}{q}}  &
=\left(  \frac{1}{\left\Vert \omega\right\Vert _{1}}%
{\displaystyle\int_{\Omega\setminus\Omega^{\prime}}}
\left\vert v\right\vert ^{q}\omega\mathrm{d}x\right)  ^{\frac{1}{q}}\\
&  \leq\left\Vert v\right\Vert _{\infty}\left(  \frac{1}{\left\Vert
\omega\right\Vert _{1}}%
{\displaystyle\int_{\Omega\setminus\Omega^{\prime}}}
\omega\mathrm{d}x\right)  ^{\frac{1}{q}}\rightarrow0,\quad\mathrm{as}\quad
q\rightarrow0^{+}.
\end{align*}
This is what happens when $\omega\equiv1,$ but in this case it is possible to
built (see \cite{EP}) a suitable function $\varphi$ that vanishes only on
$\partial\Omega$ and satisfies%
\[
\limsup_{q\rightarrow0^{+}}\left(  \left\vert \Omega\right\vert ^{-1}%
{\displaystyle\int_{\Omega}}
\left\vert \varphi\right\vert ^{q}\mathrm{d}x\right)  ^{\frac{1}{q}}>0.
\]

A simpler situation where (\ref{lim+}) holds is when $\omega$ is compactly
supported in $\Omega.$ In fact, if there exists a subdomain $\Omega^{\prime
}\subset\Omega$ such that $\omega(x)=0$ for almost every $x\in\Omega
\setminus\Omega^{\prime},$ then we can take a smooth function $\varphi\in
W_{0}^{s,p}(\Omega)$ such that $\inf_{\Omega^{\prime}}\left\vert
\varphi\right\vert =m>0$ in order to obtain%
\[
\left(  \frac{1}{\left\Vert \omega\right\Vert _{1}}%
{\displaystyle\int_{\Omega}}
\left\vert \varphi\right\vert ^{q}\omega\mathrm{d}x\right)  ^{\frac{1}{q}%
}=\left(  \frac{1}{\left\Vert \omega\right\Vert _{1}}%
{\displaystyle\int_{\Omega^{\prime}}}
\left\vert \varphi\right\vert ^{q}\omega\mathrm{d}x\right)  ^{\frac{1}{q}}\geq
m.
\]

Our feeling is that, in fact, $\mu<\infty$ whenever $\omega\in L^{r}(\Omega),$
with $r>1.$ But we were not able to prove this generically, even knowing that
\[
\mu=\infty\Longrightarrow%
{\displaystyle\int_{\Omega}}
(\log\left\vert \varphi\right\vert )\omega\mathrm{d}x=-\infty,\quad
\forall\,v\in W_{0}^{s,p}(\Omega),
\]
as (\ref{Saw1}) and Remark \ref{monota} show. Since this issue of generically
determining the finiteness of $\mu$ goes beyond of our purposes in this paper,
we will assume from now on that $\mu<\infty.$

\begin{theorem}
Let $\omega\in L^{r}(\Omega),$ $r>1,$ and suppose that $\mu<\infty.$ We have
\begin{equation}
\mu\exp\left(  \frac{p}{\left\Vert \omega\right\Vert _{1}}%
{\displaystyle\int_{\Omega}}
(\log\left\vert v\right\vert )\omega\mathrm{d}x\right)  \leq\left[  v\right]
_{s,p}^{p},\quad\forall\,v\in W_{0}^{s,p}(\Omega). \label{Sobalfa1}%
\end{equation}

\end{theorem}

\begin{proof}
The proof follows immediately from (\ref{Saw1}), by making $\alpha
\rightarrow1^{-}.$
\end{proof}

We are proceeding in the direction of proving that (\ref{Sobalfa1}) becomes an
equality for some $V\in W_{0}^{s,p}(\Omega).$ More precisely, we will show
that $\mu$ is the minimum of the functional $v\mapsto\left[  v\right]
_{s,p}^{p}$ on the set
\[
\mathcal{M}:=\left\{  v\in W_{0}^{s,p}(\Omega):%
{\displaystyle\int_{\Omega}}
(\log\left\vert v\right\vert )\omega\mathrm{d}x=0\right\}  .
\]
Note that $\mathcal{M}\not =\emptyset$ if, and only if, there exists
$\varphi\in W_{0}^{s,p}(\Omega)$ satisfying (\ref{log+}). Moreover, if
$\mathcal{M}\not =\emptyset$ then $\mu<\infty.$ The reciprocal of this will
follow immediately from our next theorem.

In the following results $V_{\alpha}$ denotes the function defined by%
\begin{equation}
V_{\alpha}:=\left\Vert \omega\right\Vert _{1}^{\frac{1}{1-\alpha}}U_{\alpha}
\label{Valfa}%
\end{equation}
where $U_{\alpha}$ is given by (\ref{Ualfa}).

It is simple to check that%
\begin{equation}
\left(  \frac{1}{\left\Vert \omega\right\Vert _{1}}%
{\displaystyle\int_{\Omega}}
\left\vert V_{\alpha}\right\vert ^{1-\alpha}\omega\mathrm{d}x\right)
^{\frac{p}{1-\alpha}}=1 \label{intValfa}%
\end{equation}
and that
\begin{equation}
\left\{
\begin{array}
[c]{lcl}%
(-\Delta_{p})^{s}\,V_{\alpha}=\left\Vert \omega\right\Vert _{1}^{\frac
{p}{1-\alpha}}\dfrac{\lambda_{\alpha}}{\left\Vert \omega\right\Vert _{1}%
}\dfrac{\omega}{(V_{\alpha})^{\alpha}} & \mathrm{in} & \Omega\\
V_{\alpha}=0 & \mathrm{on} & \mathbb{R}^{N}\setminus\Omega.
\end{array}
\right.  \label{eqValfa}%
\end{equation}

\begin{theorem}
Let $\omega\in L^{r}(\Omega),$ $r>1,$ and suppose that $\mu<\infty.$ Then
$V_{\alpha}$ converges in $W_{0}^{s,p}(\Omega)$ to a nonnegative function
$V\in\mathcal{M},$ which minimizes the functional $v\mapsto\left[  v\right]
_{s,p}^{p}$ on $\mathcal{M}.$ Moreover, the only minimizers of this functional
on $\mathcal{M}$ are $-V$ and $V.$ Consequently, the equality in
(\ref{Sobalfa1}) holds for some $v\in W_{0}^{s,p}(\Omega)$ if, and only if,
$v=kV$ for some constant $k.$
\end{theorem}

\begin{proof}
Multiplying the equation in (\ref{eqValfa}) by $V_{\alpha}$ and integrating
over $\Omega$ we obtain%
\begin{align*}
\left[  V_{\alpha}\right]  _{s,p}^{p}  &  =\left\Vert \omega\right\Vert
_{1}^{\frac{p}{1-\alpha}}\dfrac{\lambda_{\alpha}}{\left\Vert \omega\right\Vert
_{1}}%
{\displaystyle\int_{\Omega}}
(V_{\alpha})^{1-\alpha}\omega\mathrm{d}x\\
&  =\left\Vert \omega\right\Vert _{1}^{\frac{p}{1-\alpha}}\lambda_{\alpha}%
{\displaystyle\int_{\Omega}}
(U_{\alpha})^{1-\alpha}\omega\mathrm{d}x=\lambda_{\alpha}\left\Vert
\omega\right\Vert _{1}^{\frac{p}{1-\alpha}}.
\end{align*}

Therefore,
\[
\lim_{\alpha\rightarrow1^{-}}\left[  V_{\alpha}\right]  _{s,p}^{p}=\mu.
\]

This fact implies that there exist $\alpha_{n}\rightarrow1^{-}$ and a function
$V\in W_{0}^{s,p}(\Omega)$ such that: $V_{\alpha_{n}}\rightharpoonup V$
(weakly) in $W_{0}^{s,p}(\Omega),$ $V_{\alpha_{n}}\rightarrow V$ in
$L^{1}(\Omega)$ and $V_{\alpha_{n}}(x)\rightarrow V(x)$ for almost every
$x\in\Omega.$ We remark that $V\geq0$ in $\Omega$ since $V_{\alpha_{n}}>0$ in
$\Omega.$

The weak convergence implies that%
\begin{equation}
\left[  V\right]  _{s,p}^{p}\leq\lim_{n\rightarrow\infty}\left[  V_{\alpha
_{n}}\right]  _{s,p}^{p}=\mu. \label{U<mu}%
\end{equation}

Note from (\ref{Sobalfa1}) that $\mu\leq\left[  v\right]  _{s,p}^{p}$ for
every $v\in\mathcal{M}.$ Thus, by taking (\ref{U<mu}) into account, in order
to conclude that $V$ minimizes the functional $v\mapsto\left[  v\right]
_{s,p}^{p}$ on $\mathcal{M}$ we need only to prove that $V\in\mathcal{M}.$

According to (\ref{Saw1}), we have%
\[
\lambda_{t}\left\Vert \omega\right\Vert _{1}^{\frac{p}{1-t}}\left(  \frac
{1}{\left\Vert \omega\right\Vert _{1}}%
{\displaystyle\int_{\Omega}}
\left\vert V\right\vert ^{1-t}\omega\mathrm{d}x\right)  ^{\frac{p}{1-t}}%
\leq\left[  V\right]  _{s,p}^{p},\quad\forall\,t\in\lbrack\alpha_{0},1),
\]
so that
\begin{equation}
\mu\lim_{t\rightarrow1^{-}}\left(  \frac{1}{\left\Vert \omega\right\Vert _{1}}%
{\displaystyle\int_{\Omega}}
\left\vert V\right\vert ^{1-t}\omega\mathrm{d}x\right)  ^{\frac{p}{1-t}}%
\leq\left[  V\right]  _{s,p}^{p}. \label{mu1<U}%
\end{equation}
Hence, in view of (\ref{U<mu}), we can conclude that
\begin{equation}
\lim_{t\rightarrow1^{-}}\left(  \frac{1}{\left\Vert \omega\right\Vert _{1}}%
{\displaystyle\int_{\Omega}}
\left\vert V\right\vert ^{1-t}\omega\mathrm{d}x\right)  ^{\frac{p}{1-t}}\leq1.
\label{U<1}%
\end{equation}

Now, let us fix an arbitrary $t\in(\alpha_{0},1).$ Then, for all $n$ large
enough (such that $a_{0}<t<\alpha_{n}$) we have%
\[
1=\left(  \frac{1}{\left\Vert \omega\right\Vert _{1}}%
{\displaystyle\int_{\Omega}}
\left\vert V_{\alpha_{n}}\right\vert ^{1-\alpha_{n}}\omega\mathrm{d}x\right)
^{\frac{p}{1-\alpha_{n}}}\leq\left(  \frac{1}{\left\Vert \omega\right\Vert
_{1}}%
{\displaystyle\int_{\Omega}}
\left\vert V_{\alpha_{n}}\right\vert ^{1-t}\omega\mathrm{d}x\right)
^{\frac{p}{1-t}},
\]
according to (\ref{intValfa}) and Lemma \ref{monot}. It is straightforward to
check that the convergence $V_{\alpha_{n}}\rightarrow V$ in $L^{1}(\Omega)$
implies that%
\[
\lim_{n\rightarrow\infty}\left(  \frac{1}{\left\Vert \omega\right\Vert _{1}}%
{\displaystyle\int_{\Omega}}
\left\vert V_{\alpha_{n}}\right\vert ^{1-t}\omega\mathrm{d}x\right)
^{\frac{p}{1-t}}=\left(  \frac{1}{\left\Vert \omega\right\Vert _{1}}%
{\displaystyle\int_{\Omega}}
\left\vert V\right\vert ^{1-t}\omega\mathrm{d}x\right)  ^{\frac{p}{1-t}}.
\]

Therefore,
\[
1\leq\left(  \frac{1}{\left\Vert \omega\right\Vert _{1}}%
{\displaystyle\int_{\Omega}}
\left\vert V\right\vert ^{1-t}\omega\mathrm{d}x\right)  ^{\frac{p}{1-t}}.
\]
This fact and (\ref{U<1}) show that%
\[
\lim_{t\rightarrow1^{-}}\left(  \frac{1}{\left\Vert \omega\right\Vert _{1}}%
{\displaystyle\int_{\Omega}}
\left\vert V\right\vert ^{1-t}\omega\mathrm{d}x\right)  ^{\frac{1}{1-t}}=1.
\]

Since
\[
\lim_{t\rightarrow1^{-}}\left(  \frac{1}{\left\Vert \omega\right\Vert _{1}}%
{\displaystyle\int_{\Omega}}
\left\vert V\right\vert ^{1-t}\omega\mathrm{d}x\right)  ^{\frac{1}{1-t}}%
=\exp\left(  \frac{1}{\left\Vert \omega\right\Vert _{1}}%
{\displaystyle\int_{\Omega}}
(\log\left\vert V\right\vert )\omega\mathrm{d}x\right)
\]
we conclude that
\[%
{\displaystyle\int_{\Omega}}
(\log\left\vert V\right\vert )\omega\mathrm{d}x=0,
\]
that is, $V\in\mathcal{M}.$ Thus, we have
\begin{equation}
\left[  V\right]  _{s,p}^{p}=\lim_{n\rightarrow\infty}\left[  V_{\alpha_{n}%
}\right]  _{s,p}^{p}=\mu=\min_{v\in\mathcal{M}}\left[  v\right]  _{s,p}^{p}.
\label{Teo1b}%
\end{equation}

The (strong) convergence $V_{\alpha_{n}}\rightarrow V$ in $W_{0}^{s,p}%
(\Omega)$ then stems from the first equality in (\ref{Teo1b}).

Now, let $\Phi$ be a function that attains the minimum $\mu$ on $\mathcal{M}.$
We emphasize that $\Phi$ does not change sign in $\Omega.$ Otherwise, since
$\left\vert \Phi\right\vert $ also belongs to $\mathcal{M},$ we would arrive
at the contradiction%
\[
\left[  \left\vert \Phi\right\vert \right]  _{s,p}^{p}<\left[  \Phi\right]
_{s,p}^{p}=\mu\leq\left[  \left\vert \Phi\right\vert \right]  _{s,p}^{p}.
\]
Thus, without loss of generality, we will assume that $\Phi\geq0.$

Repeating the arguments developed in the proof of Proposition \ref{uniqmin} we
obtain
\begin{align*}
\lim_{\alpha\rightarrow1^{-}}\left(  \frac{1}{\left\Vert \omega\right\Vert
_{1}}\int_{\Omega}\left(  \frac{V}{2}+\frac{\Phi}{2}\right)  ^{1-\alpha}%
\omega\mathrm{d}x\right)  ^{\frac{1}{1-\alpha}}  &  \geq\lim_{\alpha
\rightarrow1^{-}}\left(  \frac{1}{\left\Vert \omega\right\Vert _{1}}%
\int_{\Omega}\left(  \frac{V}{2}\right)  ^{1-\alpha}\omega\mathrm{d}x\right)
^{\frac{1}{1-\alpha}}\\
&  +\lim_{\alpha\rightarrow1^{-}}\left(  \frac{1}{\left\Vert \omega\right\Vert
_{1}}\int_{\Omega}\left(  \frac{\Phi}{2}\right)  ^{1-\alpha}\omega
\mathrm{d}x\right)  ^{\frac{1}{1-\alpha}}\\
&  =\frac{1}{2}+\frac{1}{2}=1,
\end{align*}
so that
\[
\exp\left(  \frac{p}{\left\Vert \omega\right\Vert _{1}}%
{\displaystyle\int_{\Omega}}
\left(  \log(\frac{V}{2}+\frac{\Phi}{2})\right)  \omega\mathrm{d}x\right)
\geq1.
\]

Therefore,
\begin{align*}
\mu &  \leq\mu\exp\left(  \frac{p}{\left\Vert \omega\right\Vert _{1}}%
{\displaystyle\int_{\Omega}}
\left(  \log(\frac{V}{2}+\frac{\Phi}{2})\right)  \omega\mathrm{d}x\right) \\
&  \leq\left[  \frac{V}{2}+\frac{\Phi}{2}\right]  _{s,p}^{p}\\
&  \leq\left(  \left[  \frac{V}{2}\right]  _{s,p}+\left[  \frac{\Phi}%
{2}\right]  _{s,p}\right)  ^{p}=\left(  \frac{\mu^{\frac{1}{p}}}{2}+\frac
{\mu^{\frac{1}{p}}}{2}\right)  ^{p}=\mu
\end{align*}
from what follows that
\[
\mu^{\frac{1}{p}}=\left[  \frac{V}{2}+\frac{\Phi}{2}\right]  _{s,p}=\left(
\frac{\left[  V\right]  _{s,p}}{2}+\frac{\left[  \Phi\right]  _{s,p}}%
{2}\right)  .
\]

The strict convexity of the Gagliardo semi-norm then implies that $V=\Phi.$

Since $V$ is the unique nonnegative function that attains the minimum $\mu$ on
$\mathcal{M}$ we can conclude that the convergence $V_{\alpha_{n}}\rightarrow
V$ in $W_{0}^{s,p}(\Omega)$ does not depend on the subsequence $\alpha_{n}$
going to $1^{-}.$
\end{proof}

We would like to pass to the limit in (\ref{eqValfa}), as $\alpha
\rightarrow1^{-},$ in order to conclude that the minimizer $V$ is the solution
of the singular problem
\begin{equation}
\left\{
\begin{array}
[c]{lcl}%
(-\Delta_{p})^{s}\,u=\dfrac{\mu}{\left\Vert \omega\right\Vert _{1}}%
\dfrac{\omega}{u} & \mathrm{in} & \Omega\\
u=0 & \mathrm{on} & \mathbb{R}^{N}\setminus\Omega.
\end{array}
\right.  \label{eqV}%
\end{equation}
The convergence $V_{\alpha}\rightarrow V$ in $W_{0}^{s,p}(\Omega)$ shows that
\begin{equation}
\lim_{\alpha\rightarrow1^{-}}\left\langle (-\Delta_{p})^{s}\,V_{\alpha
},\varphi\right\rangle =\left\langle (-\Delta_{p})^{s}\,V,\varphi\right\rangle
,\quad\forall\,\varphi\in W_{0}^{s,p}(\Omega). \label{eqVa}%
\end{equation}
However, due the singular nature of the equation in (\ref{eqValfa}), this
convergence is not enough to directly obtain
\begin{equation}
\lim_{\alpha\rightarrow1^{-}}%
{\displaystyle\int_{\Omega}}
\dfrac{\omega\varphi}{(V_{\alpha})^{\alpha}}\mathrm{d}x=%
{\displaystyle\int_{\Omega}}
\dfrac{\omega\varphi}{V}\mathrm{d}x,\quad\forall\,\varphi\in C_{c}^{\infty
}(\Omega). \label{eqVb}%
\end{equation}
For this, we will assume that $r>\max\left\{  1,\frac{N}{sp}\right\}  $ in
order to use the boundedness results of Subsection \ref{secbounds}.

In the sequel $\psi\in W_{0}^{s,p}(\Omega)\cap C^{\beta_{s}}(\overline{\Omega
})$ is the function satisfying (\ref{Psi}).

\begin{lemma}
Let $\omega\in L^{r}(\Omega),$ with $r>\max\left\{  1,\frac{N}{sp}\right\}  ,$
and suppose that $\mu<\infty.$ There exist positive constants $m$ and $M$ such
that
\begin{equation}
0<m\psi\leq V_{\alpha}\leq\left\Vert V_{\alpha}\right\Vert _{\infty}\leq
M\quad\mathrm{in}\ \Omega,\quad\forall\ \alpha\in\lbrack\alpha_{0},1).
\label{lowUa}%
\end{equation}

\end{lemma}

\begin{proof}
Since $V_{\alpha}$ satisfies (\ref{eqValfa}), we can apply Theorem
\ref{global} (with $\omega$ replaced by $\left\Vert \omega\right\Vert
_{1}^{\frac{p}{1-\alpha}-1}\lambda_{\alpha}\omega$) to conclude that%
\begin{align*}
\left\Vert V_{\alpha}\right\Vert _{\infty}  &  \leq C_{\alpha}\left(
\frac{\left\Vert \omega\right\Vert _{1}^{\frac{p}{1-\alpha}-1}\lambda_{\alpha
}\left\Vert \omega\right\Vert _{r}}{S_{\theta}}\right)  ^{\frac{1}{p-1+\alpha
}}2^{\frac{b(p-1)}{(b-1)(p-1+\alpha)}}\left\vert \Omega\right\vert
^{\frac{(b-1)(p-1)}{\theta(p-1+\alpha)}}\\
&  \leq C_{\alpha}\left(  \frac{\mu\left\Vert \omega\right\Vert _{r}%
}{S_{\theta}\left\Vert \omega\right\Vert _{1}}\right)  ^{\frac{1}{p-1+\alpha}%
}2^{\frac{b(p-1)}{(b-1)(p-1+\alpha)}}\left\vert \Omega\right\vert
^{\frac{(b-1)(p-1)}{\theta(p-1+\alpha)}}%
\end{align*}
where $pr^{\prime}<\theta\leq p_{s}^{\star}$ (the equality only in the case
$sp<N$) and%
\[
C_{\alpha}:=\left(  \frac{\alpha}{p-1}\right)  ^{\frac{p-1}{p-1+\alpha}%
}\left(  1+\frac{p-1}{\alpha}\right)  \quad\mathrm{and}\quad b:=(\frac{\theta
}{r^{\prime}}-1)\frac{1}{p-1}>1.
\]

Therefore,%
\[
\limsup_{\alpha\rightarrow1^{-}}\left\Vert V_{\alpha}\right\Vert _{\infty}\leq
p\left(  \frac{1}{p-1}\right)  ^{\frac{p-1}{p}}\left(  \frac{\mu\left\Vert
\omega\right\Vert _{r}}{S_{\theta}\left\Vert \omega\right\Vert _{1}}\right)
^{\frac{1}{p}}2^{\frac{b(p-1)}{(b-1)p}}\left\vert \Omega\right\vert
^{\frac{(b-1)(p-1)}{\theta p}}.
\]
It follows that, by increasing $\alpha_{0}$ if necessary, there exists $M$
such that $\left\Vert V_{\alpha}\right\Vert _{\infty}\leq M$ for all
$\alpha\in\lbrack\alpha_{0},1).$

Thus,
\begin{align*}
\left(  -\Delta_{p}\right)  ^{s}V_{\alpha}  &  =\left\Vert \omega\right\Vert
_{1}^{\frac{p}{1-\alpha}-1}\lambda_{\alpha}\dfrac{\omega}{(V_{\alpha}%
)^{\alpha}}\\
&  \geq\left\Vert \omega\right\Vert _{1}^{\frac{p}{1-\alpha}-1}\lambda
_{\alpha}\dfrac{\omega_{1}}{(V_{\alpha})^{\alpha}}\\
&  \geq\left\Vert \omega\right\Vert _{1}^{\frac{p}{1-\alpha_{0}}-1}%
\lambda_{\alpha_{0}}\dfrac{\omega_{1}}{M^{\alpha}}\geq m^{p-1}\omega
_{1}=\left(  -\Delta_{p}\right)  ^{s}\left[  m\psi\right]  ,
\end{align*}
where $\omega_{1}=\min\left\{  \omega,1\right\}  $ and
\[
m:=\min_{\alpha_{0}\leq\alpha\leq1}\left(  \left\Vert \omega\right\Vert
_{1}^{\frac{p}{1-\alpha_{0}}-1}\lambda_{\alpha_{0}}M^{-\alpha}\right)
^{\frac{1}{p-1}}>0.
\]
Therefore, by the weak comparison principle we get the estimate%
\[
V_{\alpha}\geq m\psi>0,
\]
valid in $\Omega$, for every $\alpha\in\lbrack\alpha_{0},1).$
\end{proof}

\begin{proposition}
Let $\omega\in L^{r}(\Omega),$ with $r>\max\left\{  1,\frac{N}{sp}\right\}  ,$
and suppose that $\mu<\infty.$ The minimizer $V$ is the weak solution of the
singular problem (\ref{eqV}).
\end{proposition}

\begin{proof}
We recall that $\psi$ is positive in $\Omega$ and belongs to $C^{\beta_{s}%
}(\overline{\Omega})$ for some $0<\beta_{s}<1.$ Hence, according to the
previous lemma, $V_{\alpha}$ is bounded from below by a positive constant
(that is uniform with respect to $\alpha$) in each proper subdomain
$\Omega^{\prime}\subset\Omega.$ This property guarantees that (\ref{eqVb})
holds. Since we have already obtained (\ref{eqVa}), the conclusion follows.
\end{proof}

\section{Acknowledgments}

This work was supported by CNPq/Brazil (483970/2013-1 and 306590/2014-0) and
Fapemig/Brazil (APQ-03372-16).

\end{document}